\documentclass{amsart}
\usepackage{amssymb,epsfig}
\usepackage{mathtools}
\usepackage[normalem]{ulem}
\usepackage{calc}
\usepackage[makeroom]{cancel}

\usepackage[dvipsnames]{xcolor}

\usepackage{bm}
\usepackage{bbold}
\usepackage{stmaryrd}

\newtheorem{theorem}{Theorem}[section]
\newtheorem{lemma}[theorem]{Lemma}
\newtheorem{proposition}[theorem]{Proposition}
\newtheorem{corollary}[theorem]{Corollary}

\newtheorem{conjecture}[theorem]{Conjecture}

\theoremstyle{definition}
\newtheorem{definition}[theorem]{Definition}

\theoremstyle{remark}
\newtheorem{remark}[theorem]{Remark}

\numberwithin{equation}{section}

\newcommand{\eps}{\varepsilon}

\newcommand{\bbf}{\mathbf{b}}
\newcommand{\bn}{\mathbf{n}}
\newcommand{\dx}{d{\bf x}}

\newcommand{\Lis}{\cL\mathrm{is}}

\newcommand{\new}[1]{{\color{black}{#1}}}

\newcommand{\outcommented}[1]{}

\newcommand{\bb}{{\bf b}}
\newcommand{\cc}{c}
\newcommand{\dd}{d}
\newcommand{\EE}{{\mathcal E}_\tria}
\newcommand{\DD}{\mathbb{D}}

\DeclareMathOperator*{\argmin}{argmin}

\DeclareMathOperator{\supp}{supp}

\DeclareMathOperator{\diam}{diam}
\DeclareMathOperator{\ran}{ran}
\DeclareMathOperator{\divv}{div}

\newcommand{\tvert}{\vert\!\vert\!\vert}

\newcommand{\R}{\mathbb R}
\newcommand{\N}{\mathbb N}

\newcommand{\U}{\mathbb U}
\newcommand{\F}{\mathbb F}
\newcommand{\V}{\mathbb V}

\newcommand{\T}{\mathbb T}

\newcommand{\cL}{\mathcal{L}}

\newcommand{\TtT}{t}

\newcommand{\cB}{\mathcal B}
\newcommand{\cP}{\mathcal P}
\newcommand{\cM}{\mathcal{M}}

\newcommand{\tria}{\mathcal T}

\newcommand{\bx}{\mathbf{x}}
\newcommand{\bs}{\mathbf{s}}

\newcommand{\by}{\mathbf{y}}

\newcommand{\refs}{{r}}

\newcommand{\GG}{{G}}

\newcommand{\brR}{\breve{R}}
\def\Tr{\mathcal{T}}
\def\Trs{{\mathcal{T}_s}}
\def\tTr{\tilde{\mathcal{T}}}
\def\tTrs{{\tilde{\mathcal{T}}_s}}
\newcommand{\sll}{\langle\!\langle}
\newcommand{\srr}{\rangle\!\rangle}

\newcommand{\uPG}{u_\tria^\delta}
\newcommand{\wPG}{w_\tria^\delta}
\newcommand{\tuPG}{u_{\tilde{\tria}}^\delta}
\newcommand{\twPG}{w_{\tilde{\tria}}^\delta}

\newcommand{\be}{\begin{equation}}
\newcommand{\ee}{\end{equation}}

\newcommand{\brbrR}{\breve{\raisebox{0pt}[8pt][0pt]{$\breve{\raisebox{0pt}[6.3pt][0pt]{$R$}}$}}}

\newcommand{\bbb}{\bb^\circ}

\newenvironment{Array}{
 \everymath{\displaystyle\everymath{}}
 
 \array
 }
{\endarray }

\DeclareFontFamily{U}{mathx}{\hyphenchar\font45}
\DeclareFontShape{U}{mathx}{m}{n}{
      <5> <6> <7> <8> <9> <10>
      <10.95> <12> <14.4> <17.28> <20.74> <24.88>
      mathx10
      }{}
\DeclareSymbolFont{mathx}{U}{mathx}{m}{n}
\DeclareFontSubstitution{U}{mathx}{m}{n}
\DeclareMathAccent{\widecheck}{0}{mathx}{"71}
\DeclareMathAccent{\wideparen}{0}{mathx}{"75}

\title{Adaptive Strategies for Transport Equations}\thanks{
Both authors have been supported in part by NSF Grant  DMS 1720297. In addition, 
the first author has been supported in
part by 
the DFG Research Group 1779,  and by the SmartState and Williams-Hedberg Foundation}

\date{\today}

\author{W. Dahmen, R.P. Stevenson}
\address{Mathematics Department, University of South Carolina, Columbia, SC 29208, USA }
\email{dahmen@math.sc.edu}
\address{
Korteweg-de Vries Institute for Mathematics,
University of Amsterdam,
P.O. Box 94248,
1090 GE Amsterdam, The Netherlands}
\email{r.p.stevenson@uva.nl}



\begin{document}

\begin{abstract} 
This paper is concerned with a posteriori error bounds for linear transport equations and related questions of contriving corresponding adaptive solution strategies in the context of Discontinuous-Petrov-Galerkin schemes. After indicating our motivation for this investigation in a wider context the first major part of the paper is devoted to the derivation and analysis of a posteriori error bounds \new{that, under mild conditions on variable convection fields, are efficient and, modulo a data-oscillation term, reliable.} In particular, it is shown that these error estimators are computed at a cost that stays uniformly proportional to the problem size. The remaining part of the paper is then concerned with the question whether typical bulk criteria known from adaptive strategies for elliptic problems entail a fixed error reduction rate also in the context of transport equations. This turns out to be significantly more difficult than for elliptic problems and at this point we can give a  complete   affirmative answer  for a single spatial dimension. For the  general multidimensional case we provide partial results which we find of interest in their own right. An essential distinction from known concepts is that global arguments enter the issue of error reduction. An important ingredient of the underlying analysis, which is perhaps interesting in its own right, is to relate the derived error indicators to the residuals that naturally arise in related least squares formulations. This reveals a close interrelation between both settings regarding error reduction in the context of adaptive refinements.
\end{abstract}

\maketitle

\section{Introduction}\label{sec:1}
\paragraph{\bf Motivation, Goals:}
Adaptive solution concepts form an important component in strategies for ever advancing computational frontiers by generating
discretizations whose solutions have a desired quality (e.g. in terms of accuracy) at the expense of a possibly small problem
size, viz. number of degrees of freedom. Guaranteeing a certain performance and certifying the solution quality poses intrinsic
mathematical challenges that have triggered  numerous investigations.

It is fair to say that the most workable starting point for an adaptive method is a {\em variational formulation} of the problem at hand
that allows one to relate errors - involving the unknown solution - to residuals - involving only known quantities. A little wrinkle lies in the
fact that these residuals have to be typically evaluated in {\em dual norms} that are not  straightforward to compute. 
A first important goal is therefore
(A) to evaluate or approximate these residual quantities in a tight fashion, see e.g. the fundamental work of Verf\"urth \cite{307}.
By {\em tight} we mean in what follows that \new{modulo a data oscillation term} the {\em a posteriori bounds} are reliable as wells as efficient, i.e., up to moderate constant
multiples provide upper as well as lower bounds for the error \new{plus data oscillation}.
This by itself is important since it allows one to quantify the solution accuracy for a given discretization without a priori knowledge
about the solution such as norms of its derivatives. Aside from minimizing the size of discrete problems for a given target accuracy 
via adaptive strategies based
on such error bounds, the availability of certified bounds is essential in a {\em nested iteration} context which is sometimes the only viable 
strategy for obtaining quantifiable results within a given computational budget.

As part of an adaptive strategy a second, often  mathematically even more demanding goal (B) is to contrive a suitable  {\em mesh refinement strategy}  derived from the a posteriori residual quantities and to understand its convergence and complexity properties.   
The first step in this regard is to show that  each step of such a refinement does decrease the current error by a fixed factor.
 In many works on adaptive methods this last issue is often ignored or taken for granted when using a 
``plausible'' refinement strategy based on a posteriori indicators. However, in the context  of
highly convection dominated convection diffusion problems it is shown in
  \cite{CDW12}
that \new{an} error reduction can be  delayed until full resolution
of boundary layers  is established,  despite the fact that robust efficient and reliable error estimators are used.

 Once a fixed error reduction rate is established one then estimates in a second step the increase of degrees of freedom caused by the refinement.
 \\

\paragraph{\bf Background:}
Both steps (A) and (B) are so far best understood for problems of elliptic type and their close relatives, see e.g. \cite{BDD04,NSV09,S07}. 
By this we mean, in particular, variational
formulations involving {\em isotropic} function spaces that are essentially independent of problem parameters. Moreover, these 
variational formulations appear more or less in a natural way and lead to problems that are well conditioned (on the continuous infinite dimensional
level) in a sense to be made precise later. This luxury is lost abruptly already when dealing with simple linear transport equations.
Our particular interest in the seemingly simple model of first order  steady state linear transport equations stems from the following points.
First, classical techniques for transport equations do typically not come with tight a posteriori error bounds, let alone a rigorously founded
adaptive solution strategy.
Second, linear transport equations  form a core constituent of important kinetic models whose treatment would benefit from    
the availability of tight a posteriori error bounds because they would warrant a rigorous control of nested source term iterations 
avoiding the inversion of large linear sytems which are densely populated due to global scattering operators. 
Last but not least, linear transport equations can be viewed as a limit case of convection dominated convection diffusion equations.
Thus, appropriate variational formulations are instructive for the singularly perturbed versions as well.
We are content here with the time-independent formulations since corresponding variational formulations 
 would immediately offer   space-time formulations
for the   time dependent case where initial conditions enter as "inflow-boundary conditions".

The classical footing for rigorous a posteriori bounds is a variational formulation of the underlying (infinite-dimensional) problem for
which the induced operator is an isomorphism from the trial space onto the dual of the test space. This means errors in the trial metric are
equivalent to residuals in a dual test-norm which at least in principle contains only known quantities and hence is amendable to a numerical evaluation.
For transport equations, the lack of {\em any} diffusion is well known to cause standard Galerkin formulations being extremely ill-conditioned. This
results in notoriously unstable schemes which precludes the availability of obvious tight lower and upper
a posteriori error bounds. Instead, suitable variational formulations that could give rise to tight residual a posteriori
bounds need to be {\em unsymmetric}, i.e., trial and test metrics differ from each other. 
In this regard the Discontinuous Petrov Galerkin (DPG) concept offers a promising framework
to accommodate problem classes that are not satisfactorily treated by conventional schemes, i.e., they help identifying and numerically accessing
suitable pairs of trial and test spaces.
A concise discussion of DPG methods \new{involves two} stages:
first, in contrast to ordinary DG methods it is important  to start from a {\em mesh dependent infinite dimensional}  variational formulation
which has to be shown to be uniformly inf-sup stable with respect to the underlying meshes. The proper choice of function spaces
for the bulk as well as skeleton quantities is crucial. Second, the optimal test spaces that inherit for a given {\em finite dimensional} trial space
the stability of the infinite dimensional problem are not practical. A computational version requires replacing local infinite dimensional
test-search spaces by finite dimensional ones whose size, however, determines the computational cost. There are to our knowledge
only a few results guaranteeing uniform fully discrete stability. In the DPG context, this concerns on the one hand problems of elliptic type and their close
relatives in the sense that the involved functions spaces are isotropic \cite{75.61,37.4}.
On the other hand,  we have studied in \cite{35.8566} an essentially different problem class, namely first order linear transport problems.
There, we have proposed  a {\em fully discrete} Discontinuous Petrov Galerkin (DPG) scheme for linear transport equations
with variable convection fields
that is shown to be {\em uniformly inf-sup stable} with respect to hierarchies of shape regular meshes. The perhaps most noteworthy obstruction
encountered in this context
 is the fact that the involved function spaces are {\em anisotropic} and depend on the convection field in an essential way.
This means, for instance,  that when perturbing the convection field the test spaces not only vary with respect to the norms but even as sets.
\new{This affects, in particular, the issue of data oscillation.}
Therefore, the case of variable convection fields is rather delicate and requires a very careful organization of perturbation arguments, see
\cite{35.8566}. The present work builds on the findings in \cite{35.8566}.\\

\paragraph{\bf Objectives, Results, Layout of the Paper:}
The central objectives of this paper concern 
both goals (A) and (B) for linear first order transport equations. In Section \ref{sec:2} we briefly recall the basic DPG concepts the
remaining developments will be based upon. This includes the notion of {\em projected optimal test spaces} as well as the
principal elements of error estimation with the aid of {\em lifted residuals}.

In Section \ref{Sbroken} we detail the ingredients of the transport problem and recall from \cite{35.8566}
a corresponding DPG scheme. The level of technicality observed there is in our opinion unavoidable and stems from the
three stages of the DPG concept mentioned above. To ease accessibility of the material  and fix notation we recall from \cite{35.8566} some relevant results which the subsequent discussion will build upon. 

Section \ref{sec:apost} is devoted to goal (A) the derivation of {\em efficient and reliable} (in brief "tight") a posteriori error bounds.
DPG-schemes are often perceived as providing "natural" local error indicators ready to use for adaptive refinements.
Of course, once the uniform well-posedness of the infinite-dimensional DPG-formulation has been established the error in the trial metric is 
indeed equivalent
to a Riesz-lifted residual which is in fact a sum of local terms. However, in exactly the same way as for optimal test-functions, these quantities
require solving local {\em infinite-dimensional} Galerkin problems. Again, one has to develop a practical variant using appropriate
finite-dimensional test-search spaces. To ensure a proper complexity scaling these spaces should again have a fixed uniformly bounded
finite dimension. An improper choice of such test-search spaces could result in gross under-estimation of the actual error.
Thus, the central issue here is to rigorously ensure that the so called "practical" versions using localized test-search spaces of fixed
finite dimension do actually capture the true infinite-dimensional residual well enough to quantitatively  reflect the error \new{plus a data oscillation term}.
This is done in Section \ref{sec:apost}
 for {\em variable} convection fields under the same moderate regularity conditions as used for the uniform inf-sup stability.
Again, a central issue here is a very subtle perturbation strategy that is eventually able to cope with the essential dependence 
of the test spaces on the convection field and the fact that the perturbations are only meaningful on the finite-dimensional level.

Finally, in Section \ref{sec:red} we address goal (B).  
As indicated earlier, the situation differs in essential ways from the key mechanisms that work for 
elliptic problems.   
A key obstruction, shared with least-squares methods for other problem types, is the fact that the error indicators do not
explicitly contain any power of the local meshsize. Hence, it is now far from obvious that a fixed local refinement
actually reduces the error indicator or the error itself. This means   establishing a {\em fixed error reduction} being guaranteed by
a concrete refinement strategy becomes the main issue. In fact, we anticipate that, once error reduction
is in place the analysis of the overall complexity will then follow again along more established paths.
Therefore, we concentrate in  Section \ref{sec:red} on error reduction. The main tools are carefully exploiting 
what may be called "Petrov-Galerkin orthogonality", and local piecewise polynomial approximations. The central focus point emerging from
related attempts, however, is the fact that the tight a posteriori error indicators are actual equivalent to an entirely {\em mesh-free} indicator
of least squares type. In fact, this latter indicator may be viewed as a certain "limit" of the DPG-indicators resulting from different approximate Riesz-lifts. This connection is in our opinion
of interest in its own right.
Using these concepts, we  rigorously prove that refinement strategies based on a standard bulk criterion
imply error-reduction in a single spatial dimension. For several space dimensions we formulate an analogous result 
for collections of marked cells which in certain cases are enriched in downstream direction. The necessity of such enrichments
is, however, open. 

Since the focus of this work is on revealing the intrinsic theoretical mechanisms we dispense with numerical tests but hope
that our findings offer new insight  and will prove useful for eventually extending the current state of the art. We present in Section \ref{sec:6}   some
concluding remarks addressing, in particular, the relation between DPG and least squares schemes. 

 We sometimes write $a \lesssim b$ to express that $a$ can be bounded by a fixed constant multiple of $b$ where the multiplicative factor
 is independent of the relevant parameters $a$ and $b$ may depend on. Likewise $a\eqsim b$ means that both $a\lesssim b$ and $b \lesssim a$ 
 hold.

\section{Abstract setting and preliminary observations}\label{sec:2}
Transport dominated problems are prominent instances where {\em symmetric} variational formulations - trial and test space coincide - 
fail to provide well-conditioned problems already on the continuous level. This section serves two purposes. First, we briefly recap some
preliminaries about {\em unsymmetric} Petrov-Galerkin formulations  which, in particular, {\em Discontinuous Petrov Galerkin} (DPG) schemes
are based upon. Second, we collect  some general basic facts that will be used later in the a posteriori error analysis.

\subsection{Petrov-Galerkin formulation with projected optimal test spaces} \label{SPG}
\newcommand{\Vdel}{\V^\delta}
\newcommand{\Udel}{\U^\delta}
\newcommand{\tVdel}{\tilde{\Vdel}}
\newcommand{\VDel}{\V^\Delta}
\newcommand{\UDel}{\U^\Delta}
\newcommand{\uex}{u^{\rm ex}}
\newcommand{\bbV}{\bar{\bar{\V}}}

Let $\U$, $\V$ be Hilbert spaces and $b:\U \times \V \rightarrow \R$ a continuous bilinear form, i.e., 
$$
|b(u;v)| \le C_b\|u\|_\U \|v\|_\V,\quad u\in \U, \, v\in \V.
$$
This means that  $(\cB u)(v):=b(u;v)$ induces a bounded linear operator from $\U$ to $\V'$, the normed dual  of $\V$, endowed  as usual  with the 
 norm $\|z\|_{\V'}:= \sup_{v\in\V: \|v\|_\V=1}z(v)$.   Moreover, let us assume that $\cB$ is an isomorphism which we
 express by writing $\cB \in \Lis(\U,\V')$. It is well-known that this latter property is equivalent to the validity of the inf-sup conditions
 \be
 \label{inf-sup}
 \inf_{u\in\U}\sup_{v\in\V}\frac{b(u;v)}{\|u\|_\U \|v\|_\V}\ge \gamma, \quad  \inf_{v\in\V}\sup_{u\in\U}\frac{b(u;v)}{\|u\|_\U \|v\|_\V}\ge \gamma,
 \ee
 for some $\gamma >0$. One consequence of the entailed stability is the relation
 \be
 \label{residual}
C_b^{-1}
\|f- \cB w\|_{\V'} \le \|\uex -w\|_\U \le \gamma^{-1}\|f- \cB w\|_{\V'} ,\quad w\in\U,
 \ee
where $\uex =\cB^{-1}f$ is the exact solution of the problem: find $u\in\U$ such that
 \be
 \label{exact}
 b(u;v)= f(v)\quad v\in \V.
 \ee
 Clearly, \eqref{residual} is a natural starting point for deriving a posteriori bounds.
 The tightness of such bounds depends on the condition (number)
 \be
 \label{cond}
 \kappa_{\U,\V'}(\cB):=  \|\cB\|_{\cL(\U,\V')} \|\cB^{-1}\|_{\cL(\V',\U)} \le C_b/\gamma
 \ee
  of the problem \eqref{exact} which can equivalently be expressed as the operator equation $\cB u=f$.
 
 When trying to approximate $\uex$ by some element in a finite dimensional trial space $\Udel\subset \U$ (`$\delta$' refers to `discrete') the choice of the test space becomes
 a central issue. A by now well established mechanism is to choose a so called {\em test search space} $\bar{\V}^\delta \subseteq \V$ of dimension typically larger than ${\rm dim}\,\Udel$, for which
\be
\label{inf-supdel}
\bar{\gamma}^\delta:= \inf_{0 \neq u \in \U^\delta} \sup_{0 \neq v \in \bar{\V}^\delta}\frac{b(u;v)}{\|u\|_\U \|v\|_\V}>0.
\ee
Clearly, $\bar{\V}^\delta=\V$ would yield $\bar{\gamma}^\delta =\gamma$ so that the size of $\bar{\V}^\delta$ can be viewed as the ``invested stabilization''.
 Defining then the {\em trial-to-test map} $\TtT^\delta=\TtT^\delta(\bar{\V}^\delta) \in \cL(\U,\bar{\V}^\delta)$ by
\be
\label{TTTmap}
\langle \TtT^\delta u,v\rangle_\V=b(u;v) \quad (v \in \bar{\V}^\delta),
\ee
the function $\TtT^\delta u$ is the $\V$-orthogonal projection onto $\bar{\V}^\delta$ of the {\em optimal test function} $R^{-1} \cB u$, where $R^{-1}\colon \V' \rightarrow \V$ is the inverse Riesz map (or Riesz lift). The space
\be
\label{projopt}
\V^\delta=\V^\delta(\U^\delta,\bar{\V}^\delta):=\ran \TtT^\delta|_{\U^\delta}
\ee
is called  {\em projected optimal test space} because  $\TtT^\delta u$ is the $\V$-orthogonal projection onto $\bar{\V}^\delta$ of the {\em optimal test function}''. Also note that $\frac{b(u;\TtT^\delta u)}{\|\TtT^\delta u\|_\V}=\|\TtT^\delta u\|_\V=\sup_{0 \neq v \in \bar{\V}^\delta}\frac{b(u,v)}{\|v\|_{\V}}$, so $\V^\delta$ gives the same inf-sup constant as $\bar{\V}^\delta$.  Once \eqref{inf-supdel} has been
established for $\bar{\V}^\delta$,  the problem of finding the {\em Petrov-Galerkin solution} $u^\delta=u^\delta(f,\U^\delta,\V^\delta)
\in \U^\delta$ of
\be
\label{PG}
b(u^\delta;v)=f(v) \quad (v \in \V^\delta)
\ee
is for any
$f \in \V'$  well-posed. Moreover, the solution of \eqref{PG} yields, up to a factor $C_b/\bar{\gamma}^\delta$ (bounding $\kappa_{\U,\V'}(\cB)
$)   the best approximation to $\cB^{-1} f$ from $\U^\delta$. Here and below we use the superscript $\delta$ to refer to a discretization
or better finite dimensional problems.

In summary, it would of course be highly desirable to guarantee uniform stability in $\delta$, i.e., $\gamma^\delta \ge \underline{\gamma}>0$
in \eqref{inf-supdel},  while keeping the computational work
proportional to the dimension ${\rm dim}\,\Udel$ of the trial spaces, viz. the number of degrees of freedom. This requires
a uniform bound for the test-search spaces of the form ${\rm dim}\,\new{\bar{\V}^\delta} \lesssim {\rm dim}\,\Udel$.
In \cite{35.8566} this has been shown for linear transport problems with variable convection fields which the present 
work will heavily build on, see also Section \ref{Sbroken}.

\subsection{Error estimation} \label{Serrorest}
The accuracy of the Petrov-Galerkin solution $u^\delta\in \Udel$ is, in view of \eqref{residual}, estimated from below and above by
the residual $f-\cB u^\delta$ in $\V'$ whose evaluation  would require computing the {\em supremizer}
\be 
\label{Rexact}
\langle R(u^\delta;f),v\rangle_\V = b(u^\delta;v) - f(v),\quad v\in \V,
\ee
since $\| R(u^\delta;f)\|_\V= \|f- \cB u^\delta\|_{\V'}$. We refer to $R(u^\delta;f)$ as a {\em lifted residual}.
 The exact computation of $R(u^\delta;f)$ is, of course, not possible. However,
  to obtain a quantity that is at least uniformly proportional to $\| R(u^\delta;f)\|_\V$ one can proceed as in \eqref{TTTmap}.
  
  To that end, let us first 
  suppose 
  that $f$ is contained in a finite dimensional subspace $\F^\delta$ of $\V'$ \new{with $\dim \F^\delta \eqsim \dim \U^\delta$.}
Now let $\bar{\bar{\V}}^\delta \subset \V$ be a closed subspace, that we call the  {\em lifted residual search space},
such that
\be
\label{res-inf-sup}
\bar{\bar{\gamma}}^\delta:= \inf_{\{(u,f) \in \U^\delta\times\F^\delta\colon \cB u\neq f\}}\sup_{0 \neq v \in \bar{\bar{\V}}^\delta}\frac{b(u;v)-f(v)}{\|u-\cB^{-1} f\|_\U \|v\|_\V}>0.
\ee
In analogy to \eqref{TTTmap} we then define $R^\delta=R^\delta(\bar{\bar{\V}}^\delta):\U \times \V'\rightarrow \bar{\bar{\V}}^\delta$ by
\be
\label{Rdelta}
\langle R^\delta(u;f),v\rangle_\V=b(u;v)-f(v)= b(u- \cB^{-1}f;v) \quad (v \in \bar{\bar{\V}}^\delta).
\ee
We call $R^\delta(u;f)$ the {\em projected} {\em lifted residual} since it is the $\V$-orthogonal projection of
the exact lifted residual \eqref{Rdelta} onto $\bar{\bar{\V}}^\delta$.
For $(u,f) \in \U^\delta \times \F^\delta$, it holds that
\be
\label{prcabound}
\bar{\bar{\gamma}}^\delta \|u-\cB^{-1} f \|_\U \leq \|R^\delta(u;f)\|_\V \leq \|\cB\|_{\cL(\U,\V')}\|u-\cB^{-1} f\|_\U.
\ee
Thus the quantities $\|R^\delta(u;f)\|_\V$ provide computable upper and lower bounds for the error $\|u-\cB^{-1} f \|_\U$
incurred by an approximation $u\in \Udel$ to the exact solution $\uex =\cB^{-1}f$.  

\new{Regarding stable} DPG formulations of the transport problem, Section \ref{sec:apost} is devoted
to  identifying suitable   \new{lifted residual search spaces}  $\bbV^\delta$ for which \eqref{res-inf-sup} will be shown to hold,  {\em uniformly} in $\delta$.
 \new{In order to do so, just as for $f$ we will need that the coefficients of transport problem \new{belong to} certain finite dimensional spaces with dimensions proportional to $\dim \U^\delta$.
  Consequently, for general data, i.e. \new{right hand side} $f$ \new{as well as convection and reaction} coefficients, 
 the lower bound in \eqref{prcabound} will be valid modulo a data oscillation term that measures the distance between this data and their best approximations from the aforementioned finite dimensional spaces.}

\subsection{Towards error reduction}\label{sec:2.3}
By replacing both $\bar{\V}^\delta$ and $\bar{\bar{\V}}^\delta$ by their sum $\bar{\V}^\delta+\bar{\bar{\V}}^\delta$, from here on we will assume that $\bar{\bar{\V}}^\delta=\bar{\V}^\delta$.
Then the relation 
\be
\label{equal}
R^\delta(u_1;f)-R^\delta(u_2;f)=\TtT^\delta (\new{u_1-u_2}), \quad f \in \V', \, u_1, u_2 \in \U,
\ee
follows 
directly from the definitions of $R^\delta$ and $\TtT^\delta$.

For the Petrov-Galerkin solution $u^\delta=u^\delta(f,\U^\delta,\V^\delta) \in \U^\delta$, {\em Petrov-Galerkin orthogonality}
 $\langle R^\delta(u^\delta;f),\TtT^\delta(\U^\delta)\rangle_\V=0$ yields for
{\em any} $u \in \U^\delta$,
\be
\label{PGorth}
\|R^\delta(u^\delta;f)\|_\V^2=\|R^\delta(u;f)\|_\V^2-\|\TtT^\delta (u-u^\delta)\|_\V^2.
\ee
\begin{remark}
\label{rem:PGorth}
In particular, $u^\delta$ minimizes $\|R^\delta(\cdot;f)\|_\V$  over $\U^\delta$.
 \end{remark}

\section{A variational formulation of the transport equation with broken test and trial spaces} \label{Sbroken}
For the convenience of the reader and to fix notation we briefly recall in this section the results from \cite{35.8566}
to ensure the validity of the stability relations \eqref{inf-sup} and \eqref{inf-supdel} which all subsequent developments will be based upon.
\subsection{Transport equation} \label{Stransport}
We adhere to the setting considered in \cite[Section 2]{35.8566} and let $\Omega \subset \R^n$ be a  bounded  {polytopal} domain,  ${\bf b} \in L_\infty(\divv;\Omega)$, and $c \in L_\infty(\Omega)$.
\new{Here we set $L_\infty(\divv;\Omega):=W^0_\infty(\divv;\Omega)$ where   ${\bf b}\in W^k_\infty(\divv;\Omega)$   means that both $\divv {\bf b}$ and each ${\bf b}_i$ belong to $W^k_\infty(\Omega)$.}
As usual the outflow/inflow
boundary 
$\Gamma_\pm$ is  the closure of all those points on $\partial\Omega$ for which the outward unit normal $\bn$ is well defined and
$\pm \bn\cdot \bb >0$ while $\Gamma_0 = \partial\Omega \setminus (\Gamma_- \cup \Gamma_+)$ stands for the characteristic boundary.
We consider
the transport equation 
\begin{equation} 
\label{transport}
\left\{
\begin{array}{r@{}c@{}ll}
{\bf b}\cdot \nabla u +c u&\,\,=\,\,& f &\text{ on } \Omega,\\
u&\,\,=\,\, & g &\text{ on } \Gamma_{\!-}.
\end{array}
\right.
\end{equation}
To explain in which sense $u$ is to solve \eqref{transport} the space
$$
H({\bf b};\Omega):=\{u \in L_2(\Omega)\colon{\bf b}\cdot \nabla u \in L_2(\Omega)\},
$$
equipped with  the norm $\|u\|_{H({\bf b};\Omega)}^2:=\|u\|_{L_2(\Omega)}^2+\| {\bf b}\cdot \nabla u\|_{L_2(\Omega)}^2$,
plays a crucial role. More precisely, we need to work with the closed subspaces 
$H_{0,\Gamma_\pm}(\bb;\Omega)$ obtained by taking the closure of smooth functions vanishing on $\Gamma_\pm$,
respectively, under the norm $\|\cdot\|_{H({\bf b};\Omega)}$. In fact, for $g=0$ a first canonical variational formulation of \eqref{transport} is
to find $u\in H_{0,\Gamma_-}(\bb;\Omega)$ such that
\be
\label{first}
\int_\Omega  ({\bf b} \cdot \nabla u +c u) v \,d{\bf x}=\int_\Omega  f v\,d{\bf x}
\ee
holds for all smooth test functions $v\in C^\infty(\bar\Omega)$. Alternatively, after integration by parts one looks for $u\in L_2(\Omega)$
such that
\be
\label{second}
\int_\Omega (cv - \divv v {\bf b})u \,d{\bf x}=\int_\Omega  f v -\int_{\Gamma_{\!-}} g v {\bf b}\cdot {\bf n}\,d{\bf x}
\ee
holds for all $v\in H_{0,\Gamma_+}(\bb;\Omega)$, where now the inflow boundary condition enters as a natural boundary condition.
The second summand on the right hand side vanishes of course for $g=0$ which is the case we will focus on for convenience in what follows,
see the discussion in  \cite{35.8566}.

Accordingly, these formulations induce bounded operators
\be
\label{ops}
\begin{split}
\cB\colon & u \mapsto {\bf b} \cdot \nabla u + cu  \in \cL(H_{0,\Gamma_-}(\bb;\Omega),L_2(\Omega)),\\
\cB^\ast\colon & v \mapsto c v-\divv{v {\bf b}} \in \cL(H_{0,\Gamma_+}(\bb;\Omega),L_2(\Omega)).
\end{split}
\ee
We stress that $\cB^\ast$ is the {\em formal} adjoint of $\cB$. In fact,  the ``true'' adjoint $\cB'$ would have to be considered
as an element of $\cL(L_2(\Omega),H_{0,\Gamma_-}(\bb;\Omega)')$.  Moreover, $\cB^\ast$ is the "true" adjoint of the 
transport operator considered as a mapping in $\cL(L_2(\Omega),H_{0,\Gamma_+}(\bb;\Omega)')$.
In view of these distinctions $\cB$ and $\cB^\ast$ may in general have different properties
in terms of invertibility.

Since we do not strive for identifying the weakest possible assumptions on the problem parameters under which
both mappings are invertible we adopt this in what follows as an {\em assumption} 

\begin{align} \label{14}
\cB& 
\in \Lis(H_{0,\Gamma_{\!-}}({\bf b}; \Omega),L_2(\Omega)),\\
\cB^\ast& 
 \in \Lis(H_{0,\Gamma_{\!+}}({\bf b}; \Omega),L_2(\Omega)). \label{15}
\end{align}
where $\Lis(\U,\V)$ denotes the space of linear isomorphisms from $\U$ onto $\V$
 and refer to  e.g. \cite{35.8566,58.3} 
for concrete conditions on the problem parameters under which these assumptions are valid. Assumption \eqref{14} is essential for the
stability of the subsequent DPG scheme.
Finally, we note that the true adjoint of $\cB^\ast$, in turn, 
belongs to $\cL(L_2(\Omega),H_{0,\Gamma_+}(\bb;\Omega)')$ and can be viewed as an extension of 
$\cB$ to $L_2(\Omega)$. 
\subsection{DPG formulation of \eqref{transport}} \label{SDPG}
For a polyhedral $\Omega$ let $\T$ denote an (infinite) family of partitions $\tria$ of $\bar{\Omega}$ into essentially disjoint closed $n$-simplices {that can be created from an initial partition $\tria_\bot$ by a repeated application of a refinement rule to individual $n$-simplices which splits them into 2 or more subsimplices.}
For $\tria, \tilde{\tria} \in \T$, we write $\tria \preceq \tilde{\tria}$ when $\tilde \tria$ is a refinement of $\tria$. We write $\tria \prec \tilde{\tria}$ when $\tria \preceq \tilde{\tria}$ and  $\tria \neq \tilde \tria$.
For a $n$-simplex $K$, let 
$$
\varrho_K:=\frac{\diam(K)}{\sup\{\diam(B)\colon B \text{ a ball in }K\}}
$$
denote its shape-parameter.
With $\mathfrak{T}$ denoting the set of all $n$-simplices in any partition $\tria \in \T$, we assume that these simplices 
(or briefly $\mathfrak{T}$) are (is) uniformly {\em shape regular}
 in the sense that
\be \label{25}
\varrho:=\sup_{K \in \mathfrak{T}} \varrho_K<\infty.
\ee

For each $K \in \mathfrak{T}$, we split its boundary into characteristic and in- and outflow boundaries, i.e.,  $\partial K=\partial K_0 \cup \partial K_{\!+} \cup \partial K_{\!-}$, and, for $\tria \in \T$, denote by
$$
\partial\tria :=\cup_{K \in \tria} \partial K\setminus \partial K_0
$$
  the {\em mesh skeleton}, i.e., the union of the non-characteristic  boundary portions of the elements.

Denoting by $\nabla_\tria$ the piecewise gradient operator, we consider the ``broken'' counterpart to $H(\bb;\Omega)$
$$
H({\bf b};\tria)=\{v \in L_2(\Omega)\colon{\bf b} \cdot \nabla_\tria v \in L_2(\Omega)\},
$$
 equipped with squared ``broken'' norm
 $$
 \|v\|^2_{H({\bf b};\tria)}:=\|v\|_{L_2(\Omega)}^2+\|{\bf b} \cdot \nabla_\tria v\|_{L_2(\Omega)}^2,$$
 and view the quantities living on the skeleton as elements of the space
$$
H_{0,\Gamma_{\!-}}({\bf b};\partial\tria):=\{w|_{\partial\tria}\colon w \in H_{0,\Gamma_{\!-}}({\bf b};\Omega)\},
$$
equipped with quotient norm
\be
\label{quotient}
\|\theta\|_{H_{0,\Gamma_{\!-}}({\bf b};\partial\tria)}:=\inf\{\|w\|_{H({\bf b};\Omega)}\colon\theta =w|_{\partial\tria},\,w \in H_{0,\Gamma_{\!-}}({\bf b};\Omega)\}.
\ee

For $\tria \in \T$, a {\em piecewise} integration-by-parts of the transport equation \eqref{transport}   leads to the following `mesh-dependent' (but otherwise {\em `continuous' infinite dimensional}) variational formulation:
\begin{equation}
\label{16}
\left\{
\begin{Array}{l}
\text{For } \framebox{$\U_\tria:=L_2(\Omega) \times H_{0,\Gamma_{\!-}}({\bf b};\partial\tria)$},\,\framebox{$\V_\tria:= H({\bf b};\tria)$}
, \text{given } f \in \V_\tria',\\
\text{find the solution } \framebox{$(u_\tria,\theta_\tria)$}=(u_\tria(f),\theta_\tria(f)) \in \U_\tria \text{ that, for all  } v \in \V_\tria, \text{ satisfies}\\
b_\tria(u_\tria,\theta_\tria;v) := \int_\Omega (cv-{\bf b}\cdot\nabla_\tria v -v \divv {\bf b})u_\tria \,d{\bf x}+\int_{\partial\tria} \llbracket v {\bf b} \rrbracket\theta_\tria \,d{\bf s} =f(v).
\end{Array}
\right.
\end{equation}
\new{Here $\int_{\partial\tria} \llbracket v {\bf b} \rrbracket\theta_\tria \,d{\bf s}$ should read as the unique extension to a bounded bilinear form on 
$H_{0,\Gamma_{\!-}}({\bf b};\partial\tria) \times \V_\tria$ (\cite[Lemma.~3.4]{35.8566}) of the integral over $\partial\tria$ of the product of $\llbracket v {\bf b} \rrbracket$ and  $\theta_\tria$, where for smooth $v$}
and $x \in \partial K \cap \partial K'$, 
$$
\llbracket v {\bf b} \rrbracket(x):=(v {\bf b}|_{K}\cdot{\bf n}_K)(x)+(v {\bf b}|_{K'}\cdot{\bf n}_{K'})(x),
$$
and $\llbracket v {\bf b} \rrbracket(x):=(v {\bf b}|_{K}\cdot{\bf n}_K)(x)$ for $x \in \partial\Omega \cap \partial K$.
Note the introduction of the notation $(u_\tria,\theta_\tria)$ for the exact solution of this variational problem.

In the following, we abbreviate $\|\cB^{-1}\|_{\cL(L_2(\Omega),H_{0,\Gamma_{\!-}}({\bf b};\Omega))}$, $\|(\cB^\ast)^{-1}\|_{\cL(L_2(\Omega),H_{0,\Gamma_{\!+}}({\bf b};\Omega))}$, $\|\divv {\bf b}\|_{L_\infty(\Omega)}$, $\|c\|_{L_\infty(\Omega)}$, and $\|c-\divv {\bf b}\|_{L_\infty(\Omega)}$ as $\|\cB^{-1}\|$, $\|\cB^{-\ast}\|$, $\|\divv {\bf b}\|$, $\|c\|$, and $\|c-\divv {\bf b}\|$ respectively. The following result roughly says that
\eqref{16} is uniformly inf-sup stable whenever the operators $\cB, \cB^*$ are isomorphisms on the respective function space pairs.

\begin{theorem}[{\cite[Theorem~3.1]{35.8566}}] \label{th4}
Assume that $\bbf\in L_\infty(\divv;\Omega)$, $c\in L_\infty(\Omega)$ and that 
conditions \eqref{14}, \eqref{15} hold. Then,
defining $\cB_\tria : \U_\tria \to \V_\tria'$ by $(\cB_\tria(u,\theta))(v):=b_\tria(u,\theta;v)$, 
one has $\cB_\tria \in \Lis(\U_\tria,\V_\tria')$ with
\begin{align*}
 \|\cB_\tria\|_{\cL(\U_\tria,\V_\tria')} &\leq 2+\|\divv {\bf b}\|+\|c-\divv{\bf b}\|,\\
\|\cB_\tria^{-1}\|_{\cL(\V_\tria',\U_\tria)} &\leq \sqrt{\|\cB^{-\ast}\|^2+\tilde{C}_\cB^2},
\end{align*}
where $\tilde{C}_\cB:=(1+\|\cB^{-\ast}\| (1+\|c-\divv {\bf b}\|)) 
\|\cB^{-1}\|(\|c-\divv {\bf b}\|+1)$.
\end{theorem}

The additional independent variable $\theta_\tria$ introduced in the mesh-dependent variational formulation replaces the trace $u_\tria|_{\partial\tria}$ which generally  is not defined for $u_\tria \in L_2(\Omega)$. 
\emph{If} $f \in L_2(\Omega)$, however, or, equivalently, $u_\tria \in H_{0,\Gamma_{\!-}}({\bf b};\Omega)$, \emph{then} a reversed integration by parts shows that 
$$
u_\tria= {\framebox{$\uex$}=\uex(f)}:=\cB^{-1} f,\quad \theta_\tria=\uex|_{\partial\tria}.
$$
\subsection{Petrov-Galerkin}
For any $\tria \in \T$, let $\tria_s \in \T$ be a \emph{refinement} of $\tria$.
We set 
\begin{equation} \label{26}
\sigma:=\sup_{\tria \in \T} \max_{K' \in \tria} \Big(\max_{\{K \in \tria_s\colon K \subset K'\}} \frac{\diam(K)}{\diam(K')},\diam(K')\Big),
\end{equation}
which later will be assumed to be sufficiently small. \new{We also require that
\begin{equation} \label{26a}
\inf_{\tria \in \T} \min_{K' \in \tria} \min_{\{K \in \tria_s\colon K \subset K'\}} \frac{\diam(K)}{\diam(K')}\gtrsim \sigma,
\end{equation}}
This means that we will assume that any partition $\tria \in \T$ is sufficiently fine, and, what is more important, that $\tria_s \in \T$ is a refinement of $\tria$ such that \new{the}
  {\em subgrid refinement factor} (or sometimes called subgrid refinement depth)
 $1/\sigma$ when going from any $\tria$ to $\tria_s$ is sufficiently large.
\new{In addition to the conditions from Theorem~\ref{th4}, we assume henceforth
\begin{equation} 
\label{24}
{\bf b}|_{K} \in W^1_\infty(\divv;K) \text{ and } c|_K \in W_\infty^1(K) \,\,(K \in \tria_s), \text{  and  } |{\bf b}|^{-1}\in L_\infty(\Omega).
\end{equation}}

Under these assumptions, we have  the following result:
\begin{theorem}[{\cite[Thm.~4.8]{35.8566}}] \label{th1} 
Selecting, for some \new{fixed} degrees $m_w \geq 1$, and $m_u$,
\begin{align*}
\U_\tria^\delta&:=\prod_{K' \in \tria} \cP_{m_u}(K')\times  \Big(H_{0,\Gamma_{\!-}}({\bf b};\Omega) \cap  \prod_{K' \in \tria} \cP_{m_w}(K')\Big)\Big|_{\partial\tria_s}\subset \U_{\tria_s}, \\
\bar{\V}_{\tria_s}^\delta&:=\prod_{K \in \tria_s} \cP_{m_v}(K) \subset \V_{\tria_s},
\end{align*}
where $m_v \geq \max(m_u,m_w)+1$, for $\sigma>0$ small enough it holds that
$$
\inf_{\tria \in \T}\inf_{0 \neq (u,\theta) \in \U_\tria^\delta} \sup_{0 \neq v \in \bar{\V}_{\tria_s}^\delta} \frac{b_{\tria_s}(u,\theta;v)}{\|(u,\theta)\|_{\U_{\tria_s}}\|v\|_{\V_{\tria_s}}}>0,
$$
only dependent on (upper bounds for)
$m_u$, $m_w$, $\varrho$,  $\||{\bf b}|^{-1}\|_{L_\infty(\Omega)}$, 
$\|\cB^{-1}\|_{\cL(L_2(\Omega),H_{0,\Gamma_{\!-}}({\bf b};\Omega))}$,
\new{$\sup_{K \in \tria_s}\|{\bf b}|_K\|_{W^1_\infty(\divv;K)}$, and $\sup_{K \in \tria_s}\|c|_K\|_{W^1_\infty(K)}$.\footnote{In the theorem in \cite{35.8566} the last two expressions read as $\|{\bf b}\|_{W^1_\infty(\divv;\Omega)}$ and $\|c\|_{W^1_\infty(\Omega)}$, but an inspection of the proof shows that they can be replaced by the current ones.}}
\end{theorem}

Consequently, as we have seen in Sect.~\ref{SPG}, the \emph{Petrov-Galerkin solution} $(u_\tria^\delta,\theta_\tria^\delta) \in \U_\tria^\delta \subset \U_{\tria_s}$ of
\be \label{PG1}
b_{\tria_s}(u_\tria^\delta,\theta_\tria^\delta;v)=f(v) \quad (v \in \ran \TtT^\delta|_{\U_\tria^\delta}),
\ee
where
\be \label{PG2}
\langle \TtT^\delta (u,\theta),v\rangle_{\V_{\tria_s}}=b_{\tria_s}(u,\theta;v)\quad (v \in \bar{\V}_{\tria_s}^\delta),
\ee
is a near-best approximation to 
$(u_{\tria_s},\theta_{\tria_s})=\cB_{\tria_s}^{-1} f \in \U_{\tria_s}$ from $\U_\tria^\delta$.

Since the above stability is ensured by "some" fixed subgrid-refinement depth, the computational work for computing the test-basis functions
remains uniformly proportional to the dimension of the trial space and in this sense scales optimally. While the actual depth is
hard to quantify precisely the experiments considered in \cite{35.8566} actually suggest that one or even no additional refinement suffice
in these examples.

We emphasize that although the bilinear form $b_{\tria_s}$ corresponds to the variational formulation of the transport problem obtained by applying a piecewise integration by parts w.r.t. the `{\em fine}' partition $\tria_s$, and the test search space $\bar{\V}_{\tria_s}^\delta$ consists of piecewise polynomials w.r.t. $\tria_s$ too,
the applied {\em trial} space consists of  pairs of functions that are piecewise polynomial w.r.t. the `{\em coarse}' partition $\tria$, or that are restrictions of such functions to $\partial\tria_s$, respectively.

\begin{remark} \label{rem0}
Actually, in \cite{35.8566} we established a slightly stronger inf-sup condition.
Defining
$$
\breve{\U}_\tria^\delta:=\prod_{K' \in \tria} \cP_{m_u}(K')\times  \Big(H_{0,\Gamma_{\!-}}({\bf b};\Omega) \cap  \prod_{K' \in \tria} \cP_{m_w}(K')\Big)
\subset L_2(\Omega) \times H_{0,\Gamma_{\!-}}(\bb;\Omega)=:\breve{\U},
$$
any $(u,\theta) \in \U_\tria^\delta  {\subset \U_{\tria_s}}$ is of the form $(u,w|_{\partial \tria_s})$ for some $(u,w) \in \breve{\U}_\tria^\delta$.
In \cite{35.8566} it was shown that
\be \label{27}
\inf_{\tria \in \T}\inf_{0 \neq (u,w) \in \breve{\U}_\tria^\delta} \sup_{0 \neq v \in \bar{\V}_{\tria_s}^\delta} \frac{b_{\tria_s}(u,w|_{\partial\tria_s};v)}{\|(u,w)\|_{\breve{\U}}\|v\|_{\V_{\tria_s}}}>0,
\ee
which implies Theorem~\ref{th1} because of
$\|(u,w)\|_{\breve{\U}}\geq \|(u,w|_{\partial\tria_s})\|_{\U_{\tria_s}}$.

Knowing \eqref{27}, the uniform boundedness of $\|\cB_{\tria_s}\|_{\cL(\U_{\tria_s},\V_{\tria_s}')}$ shows that $\|(u,w)\|_{\breve{\U}}\eqsim \|(u,w|_{\partial\tria_s})\|_{\U_{\tria_s}}$ on 
$\breve{\U}_\tria^\delta$. In particular this means that $(u,w|_{\partial \tria_s})$ determines $(u,w) \in \breve{\U}_\tria^\delta$ uniquely, so that equally well we can speak of the  {\emph{Petrov-Galerkin solution}
$(u_\tria^\delta,w_\tria^\delta) \in \breve{\U}_\tria^\delta$ of
\be \label{PGsol}
b_{\tria_s}(u_\tria^\delta,w_\tria^\delta|_{\tria_s};v)=f(v) \quad (v \in \ran \TtT^\delta|_{\breve{\U}_\tria^\delta}),
\ee
where, of course, $\TtT^\delta(u,w):=\TtT^\delta(u,w|_{\partial\tria_s})$.}
\end{remark}

\begin{remark}
\label{rem:nested}
The trial spaces $\breve{\U}_\tria^\delta$ are {\em nested} whenever the underlying   partitions are nested. This
plays an important role for conceiving adaptive strategies.
\end{remark}

\begin{remark}
\label{rem:int} Since a polynomial of degree $\geq 3$ is not uniquely determined by its values on the boundary of a triangle, 
the inf-sup stability \eqref{27} can apparently only hold for $m_w \geq 3$  when $\tria_s$ is a true refinement of $\tria$.

In the latter formulation involving the lifted version  $w$ of the skeleton quantity $\theta$, the scheme provides two approximations for the solution of the transport problem, 
namely $u_\tria^\delta \in L_2(\Omega)$
and a second one  $w_\tria^\delta\in H(\bb;\Omega)$.
\end{remark}
\begin{remark}
For a function in $ \prod_{K' \in \tria} \cP_{m_w}(K')$ to be in $H({\bf b};\Omega)$, it \new{has} to be continuous  at any intersection of an in- and outflow face of any $K' \in \tria$.
To realize this condition, an obvious approach is to consider in the definition of $\breve{\U}_\tria^\delta$ or $\U_\tria^\delta$ the space
$H_{0,\Gamma_{\!-}}({\bf b};\Omega) \cap C(\Omega) \cap \prod_{K' \in \tria} \cP_{m_w}(K')$ instead of $H_{0,\Gamma_{\!-}}({\bf b};\Omega) \cap  \prod_{K' \in \tria} \cP_{m_w}(K')$.
Obviously with this modification, Thm.~\ref{th1} and Remark~\ref{rem0} remain valid, and so does the whole further exposition.
\end{remark}

\section{A posteriori error estimation}\label{sec:apost}
\def\Tr{\mathcal{T}}
\def\Trs{{\mathcal{T}_s}}
\def\tTr{\tilde{\mathcal{T}}}
\def\tTrs{{\tilde{\mathcal{T}}_s}}
\new{The central goal in this section is to establish the validity of \eqref{res-inf-sup} for locally uniformly finite dimensional test search spaces
of the same form as used in Theorem \ref{th1}. We will be able to do so modulo a data oscillation term. The principal difficulty lies 
in an intrinsic sensitivity of essential problem metrics with respect to perturbations in the {\em convection field}. To exploit the fact
that we can identify optimal test spaces for locally constant convection comes at the price of an elaborate perturbation analysis to be
carried out in this section. In fact, it involves two levels of perturbation, namely passing from general data  $\bb,c,f$ to piecewise
polynomial ones, and then to piecewise constant $\bb$ on a subgrid. The passage to piecewise polynomial
data is accounted for by {\em data oscillation} terms. The piecewise polynomial structure of the data with respect to the
(coarser) trial grid $\tria$, in turn, is needed to control the effect of the reduction to piecewise constant convection on the discrete level.

\subsection{Main Results} 

\begin{theorem} \label{th2new} Assume \eqref{14}, and let $f \in L_2(\Omega)$.
For $\tria \in \T$, assume that for $K' \in \tria$, $\bb|_{K'} \in W^1_\infty(K')^n$, $c|_{K'} \in W^1_\infty(K')$, and 
let $\tilde \bb$, $\tilde c$, $\tilde f$ denote the best piecewise polynomial approximations to $\bb$, $c$, $f$ of degrees $m_\bb$, $m_c$, and $m_f$ w.r.t. $\tria$ in \new{$L_\infty(\Omega)^n$-, $L_\infty(\Omega)$-, or $L_2(\Omega)$}-norm, respectively. Let
\begin{align*}
{\rm o}&{\rm{sc}}_\tria(\bb,c,f):=\\
&\max\Big(\|f-\tilde f\|_{L_2(\Omega)},\big(\|c-\tilde c\|_{L_\infty(\Omega)},\max_{K' \in \tria} \,\diam(K')^{-1} \||\bb-\tilde \bb|\|_{L_\infty(K')}\big)\|f\|_{L_2(\Omega)}\Big),
\end{align*}
and
\begin{equation} 
\label{mv}
m_v \geq \max\big(m_w+\max(m_c,1,m_\bb-1),m_u+\max(m_c,1),m_f\big).
\end{equation}

Then, with $\breve{\U}_\tria^\delta$ and 
$\bar{\V}_{\tria_s}^\delta$ as defined before,  
for fixed sufficiently small $\sigma>0$ in \eqref{26}, and for any $(u,w) \in  \breve{\U}_\tria^\delta$ for which $\max(\|u\|_{L_2(\Omega)},\|w\|_{L_2(\Omega)}) \lesssim \|f\|_{L_2(\Omega)}$ {\rm(}which, on account of \eqref{27}, is valid for the Petrov-Galerkin solution{\rm)}, it holds that
\be \label{aposteriori}
\|R^\delta_{\tria_s}\|_{\V_{\tria_s}} \lesssim \|(u^{\rm ex},u^{\rm ex})-(u,w)\|_{\breve{\U}} \lesssim \|R^\delta_{\tria_s}\|_{\V_{\tria_s}} +{\rm osc}_\tria(\bb,c,f),
\ee
where $R^\delta_{\tria_s} \in \bar{\V}_{\tria_s}^\delta$ is defined by
\be \label{s4}
\langle R_{\tria_s}^\delta, v\rangle_{\V_{\tria_s}}=b_{\tria_s}(u,w|_{\partial\tria_s};v)-\int_\Omega f v \,d{\bf x} \quad (v \in \bar{\V}_{\tria_s}^\delta),
\ee
{\rm(}cf. \eqref{Rdelta}{\rm)}. The constants absorbed by the $\lesssim$-symbols in \eqref{aposteriori} depend only on the polynomial degrees and on (upper bounds for) $\varrho$, $\||\bb|^{-1}\|_{L_\infty(\Omega)}$, $\sup_{K' \in \tria} \|\bb\|_{W^1_\infty(K')^n}$, $\sup_{K' \in \tria} \|c\|_{W^1_\infty(K')}$, and $\|\cB^{-1}\|_{\cL(L_2(\Omega),H_{0,\Gamma_{\!-}}(\bb;\Omega))}$.
\end{theorem}

Note that for given degrees $m_u$ and $m_w$, then for sufficiently large $m_\bb$, $m_c$, and $m_f$ (and thus $m_v$) and piecewise smooth $\bb$, $c$ and $f$, ${\rm osc}_\tria(\bb,c,f)$ can be reduced at a better rate in terms of $\# \tria$ than generally can be expected for $\|(u^{\rm ex},u^{\rm ex})-(u_\tria^\delta,w_\tria^\delta)\|_{\breve{\U}}$.

The proof of Theorem~\ref{th2new} will be based on the following Proposition.

\begin{proposition} \label{lemma1new} In the situation of Theorem~\ref{th2new}, let
$$
\tilde b_{\tria_s}(u,w;v) :=  \sum_{K\in \tria_s}\tilde b_K(u,w;v),
$$
where
$$
\tilde b_K(u,w;v):=\int_K ({\tilde c} u +\tilde \bb \cdot \nabla w) v+(w-u) (v \divv {\tilde \bb}+\tilde \bb \cdot \nabla v)dx.
$$
Then for any $(u,w,\tilde f) \in \DD_\tria:=\prod_{K'\in \tria}\cP_{m_u}(K')\times \cP_{m_w}(K')\times \cP_{m_f}(K')$, it holds that
\be \label{105}
\underbrace{\|w-u\|_{L_2(\Omega)}+\|{\tilde \bb} \cdot \nabla_\tria w +{\tilde c} w-\tilde f\|_{L_2(\Omega)}}_{\textstyle \EE(u,w,\tilde f):=} \lesssim \sup_{0 \neq v \in \bar{\V}_{\tria_s}^\delta}\frac{\tilde b_{\tria_s}(u,w;v)-\int_\Omega \tilde f v \dx}{\|v\|_{\V_{\tria_s}}},
\ee
only dependent on the polynomial degrees and on (upper bounds for) $\varrho$, $\||\bb|^{-1}\|_{L_\infty(\Omega)}$, 
$\sup_{K' \in \tria} \|\bb\|_{W^1_\infty(K')^n}$, and $\sup_{K' \in \tria} \|c\|_{W^1_\infty(K')}$.
\end{proposition}

\begin{remark}
\label{rem:quadrature}
In a strict sense the quamtities $R_{\tria_s}^\delta$, defined in \eqref{s4} can,  for general coefficients $\bb, c$, not be computed {\em exactly}.
Under the presumption that the accuarcy of quadrature can be adjusted, this issue is usually neglected, as we did in \eqref{aposteriori} above. 
 Since quadrature is in essence based on replacing the integrand by a local polynomial approximation, a natural way of incorporating
 this issue here is to work with the analogous projected lifted residuals with respect to the perturbed data 
 \be
 \label{s4q}
 \langle \tilde R_{\tria_s}^\delta, v\rangle_{\V_{\tria_s}}=\tilde b_{\tria_s}(u,w|_{\partial\tria_s};v)-\int_\Omega \tilde f v \,d{\bf x} \quad (v \in \bar{\V}_{\tria_s}^\delta),
\ee
which can be computed exactly. Under the assumptions of Theorem \ref{th2new} one then obtains the following estimates
\be \label{aposterioriq}
\|\tilde R^\delta_{\tria_s}\|_{\V_{\tria_s}}  \lesssim \|(u^{\rm ex},u^{\rm ex})-(u,w)\|_{\breve{\U}} +{\rm osc}_\tria(\bb,c,f) \lesssim \|\tilde R^\delta_{\tria_s}\|_{\V_{\tria_s}} +{\rm osc}_\tria(\bb,c,f).
\ee
\end{remark}

Sections~\ref{Slifting}--\ref{Sproof} will be devoted to the proof of Proposition~\ref{lemma1new}. In the course of these developments it will be seen that 
the residual $\EE(u,w,\tilde f)$ is actually equivalent to $\|\tilde R^\delta_{\tria_s}\|_{\V_{\tria_s}}$ and may therefore also be used as 
error indicator. \\ 

Assuming for the moment the validity of Proposition~\ref{lemma1new}, we can give the proof of Theorem~\ref{th2new} and Remark \ref{rem:quadrature}.

\begin{proof}[Proof of Theorem~\ref{th2new}] Applications of the triangle inequality show that
$$
\|(u^{\rm ex},u^{\rm ex})-(u,w)\|_{\breve{\U}} \eqsim
\|u-w\|_{L_2(\Omega)} +\|u^{\rm ex}-w\|_{H(\bb;\Omega)},
$$
and it holds that $\|u^{\rm ex}-w\|_{H(\bb;\Omega)} \leq \|\cB^{-1}\|_{\cL(L_2(\Omega),H_{0,\Gamma_{\!-}}(\bb;\Omega))} \|{\mathcal B} w-f\|_{L_2(\Omega)}$.

By using the inverse inequality on piecewise polynomials of degree $m_w$, 
and $\|w\|_{L_2(\Omega)} \lesssim \|f\|_{L_2(\Omega)}$, we infer that
$$
\big|\|\cB w-f\|_{L_2(\Omega)}-\|{\tilde \bb} \cdot \nabla_\tria w +{\tilde c} w -\tilde f\|_{L_2(\Omega)} \big| \lesssim {\rm osc}_\tria(\bb,c,f).
$$
An application of Proposition~\ref{lemma1new} gives
$$
\|w-u\|_{L_2(\Omega)}+\|{\tilde \bb} \cdot \nabla_\tria w +{\tilde c} w -\tilde f\|_{L_2(\Omega)} \lesssim \sup_{0 \neq v \in \bar{\V}_{\tria_s}^\delta}\frac{{\tilde b}_{\tria_s}(u,w;v)-\int_\Omega \tilde f v \dx}{\|v\|_{\V_{\tria_s}}}
$$
We show next that the right hand side deviates from the analogous unperturbed quantity only by ${\rm osc}_\tria(\bb,c,f)$. To that end,
it holds that $|\int_\Omega f v \,dx-\int_\Omega \tilde f v \,dx| \leq \|f-\tilde f\|_{L_2(\Omega)} \|v\|_{L_2(\Omega)}$, and
\begin{align*}
&\big|b_{\tria_s}(u,w|_{\partial \tria_s};v)-{\tilde b}_{\tria_s}(u,w;v)\big|= \\
&\big|\sum_{K \in \tria_s} 
\int_K ((c-{\tilde c}) u +(\bb-\tilde \bb) \cdot \nabla w) v+(w-u) (v \divv (\bb-{\tilde \bb})+(\bb-\tilde \bb) \cdot \nabla v)dx\\
&\lesssim {\rm osc}_\tria(u,w,f) \|v\|_{L_2(\Omega)},
\end{align*}
where we used that for $K' \in \tria$
\be \label{103}
\|\bb -\tilde \bb\|_{W^1_\infty(K')^n}\lesssim \diam(K')^{-1} \|\bb -\tilde \bb\|_{L_\infty(K')^n}
\ee
(cf. e.g. \cite{34.55} for the argument);
$\|\nabla v\|_{L_2(K)^n} \lesssim \diam(K)^{-1} \|v\|_{L_2(K)}$ for $K \in \tria_s$; and  
$\diam(K)^{-1} \lesssim \diam(K')^{-1}$
for $K \subset K'$ by \eqref{26a}.
We conclude that
$$
\Big|
\sup_{0 \neq v \in \bar{\V}_{\tria_s}^\delta}\frac{{\tilde b}_{\tria_s}(u,w;v)-\int_\Omega \tilde f v \dx}{\|v\|_{\V_{\tria_s}}}
-
\sup_{0 \neq v \in \bar{\V}_{\tria_s}^\delta}\frac{b_{\tria_s}(u,w;v)-\int_\Omega f v \dx}{\|v\|_{\V_{\tria_s}}}
\Big| \lesssim {\rm osc}_\tria(\bb,c,f).
$$
From 
$$
\|R^\delta_{\tria_s}\|_{\V_{\tria_s}}=
\sup_{0 \neq v \in \bar{\V}_{\tria_s}^\delta}\frac{b_{\tria_s}(u,w;v)-\int_\Omega f v \dx}{\|v\|_{\V_{\tria_s}}}
\lesssim 
\|(u^{\rm ex},u^{\rm ex})-(u,w)\|_{\breve{\U}}, 
$$
the assertion of Theorem \ref{th2new} follows. The above argument also shows that $\|\tilde R_{\tria_s}^\delta -  R_{\tria_s}^\delta\|_{\V_{\tria_s}}
\lesssim {\rm osc}_\tria(\bb,c,f)$ which confirms \eqref{aposterioriq} as well.
\end{proof}

\begin{remark} The estimate \eqref{103} together with $\|\bb -\tilde \bb\|_{L_\infty(K')^n}\lesssim \diam(K') |\bb|_{W^1_\infty(K')^n}$ implies that
$|\tilde \bb|_{W^1_\infty(K')^n} \lesssim |\bb|_{W^1_\infty(K')^n}$ which, as the analogous 
$|\tilde c|_{W^1_\infty(K')} \lesssim |c|_{W^1_\infty(K')}$, will be often used in the following.
\end{remark}}

 \subsection{Lifting modified residuals} \label{Slifting}
 
As in \cite{35.8566} the verification of uniform inf-sup stability \eqref{105} relies on judiciously perturbing exact  Riesz lifts  corresponding 
to certain perturbed bilinear forms.
To describe this, given $\tria \in \T$, let we define for $K \in \tria_s \cup \tria$, the {\em constants}
$$
\new{\breve \bb}_K:=|K|^{-1} \int_K \new{\tilde \bb}\,d {\bf x}, \quad
\new{\breve \dd}_K:=|K|^{-1} \int_K \divv \new{\tilde \bb}\,d {\bf x},
$$
and for $\tria \in \T$, let $\new{\breve \bb}\in L_\infty(\Omega)^n$ be given by
\begin{equation} \label{202}
\new{\breve \bb}|_K:=\new{\breve \bb}_K \quad (K \in \tria_s).
\end{equation}
On $\breve{\U}^\delta_\tria \times \bar{\V}^\delta_{\tria_s}$ we introduce yet another {\em modified bilinear form}
\be \label{barbhw0}
\breve b_{\tria_s}(u,w;v) :=  \sum_{K\in \tria_s}\breve b_K(u,w;v),
\ee
where the summands $\breve b_K(u,w;v)$ are defined by
\begin{equation} \label{200}
\breve b_K(u,w;v):=\int_K \big(\new{\tilde \bb} \cdot \nabla u+\new{\tilde c} u+\new{\breve d}_K (w-u)\big) v \,d\bx+\int_{\partial K} \new{\breve \bb}_K \cdot {\bf n}_K (w-u)v \,d\bs.
\end{equation}
\new{Note that $\breve \bb$ and $\breve d$ are piecewise constant w.r.t. $\tria_s$, whereas $\tilde \bb$ and $\tilde c$ are piecewise polynomial w.r.t. $\tria$.} This form is only introduced for analysis purposes since, as it turns out, it allows us to determine local lifted residuals exactly. Their use
requires then yet another layer of perturbation arguments.

\begin{remark}
\label{rem:particular}
The particular form of the modified bilinear form \eqref{barbhw0}-\eqref{200}, in particular the integrand in the boundary integral over $\partial K$,
is to ensure  that  $v \mapsto \breve{b}_K(u,w;v)$ is in $H(\new{\breve \bb}_K;K)'$.
\end{remark}

The proof of Proposition~\ref{lemma1new} is based on the following steps: \medskip

{\bf (I)} \label{steps}We will construct a $\brbrR=\brbrR_{\tria_s}(u,w;\tilde f) \in \bar{\V}_{\tria_s}^\delta$, 
such that $\breve b_{\tria_s}(u,w;\brbrR) -\int_\Omega {\tilde f} \brbrR  \,d{\bf x}\gtrsim 
\new{\EE(u,w,\tilde f)}
 \|\brbrR\|_{\V_{\tria_s}} $, of course, uniformly in $\tria \in \T$ and  $(u,w,\tilde f) \new{\in \DD_\tria}$. \medskip

{\bf (II)} Starting from the simple decomposition
\be
\label{split}
\new{\tilde b}_{\tria_s}(u,w;\brbrR)-\int_\Omega {\tilde f} \brbrR \,d{\bf x}= \breve b_{\tria_s}(u,w;\brbrR)-\int_\Omega {\tilde f} \brbrR  \,d{\bf x}+ \new{\tilde b}_{\tria_s}(u,w;\brbrR) - \breve b_{\tria_s}(u,w;\brbrR),
\ee
we will show for the second summand that
\be
\label{wish}
|\new{\tilde b}_{\tria_s}(u,w;\brbrR) - \breve b_{\tria_s}(u,w;\brbrR)|  \le  \delta \new{\EE(u,w,\tilde f)} \|\brbrR\|_{\V_{\tria_s}},
\ee
holds for a sufficiently small $\delta >0$, depending on the inf-sup constant for the first summand. \medskip

As the construction of the modified bilinear form $\breve b_{\tria_s}$ from $\new{\tilde b}_{\tria_s}$ builds on the approximation of $\new{\tilde \bb}$ by $\new{\breve \bb}$, the space $H(\new{\breve \bb};\tria_s)=\prod_{K \in \tria_s} H(\new{\breve \bb}_K;K)$, equipped with the corresponding product norm $\|\cdot\|_{H(\new{\breve \bb};\tria_s)}$, will play its role as a space `nearby' $\V_{\tria_s}=H(\bb;\tria_s)$.
In the next proposition, we equip $H(\new{\breve \bb}_K;K)$ with an equivalent Hilbertian norm that, as we will see, gives rise to a local Riesz lift $H(\new{\breve \bb}_K;K)' \rightarrow H(\new{\breve \bb}_K;K)$ of the residual of the modified bilinear form
that can be determined {\em explicitly}.

\begin{proposition}[{\cite[Remark~4.5]{35.8566}}] \label{equivalence}
For
$$
\diam(K) \leq |\new{\breve \bb}_K|,
$$
and with
$r({\bf s})$ denoting the distance from ${\bf s} \in \partial K_{\!-}$ to $\partial K_{\!+}$ along $\new{\breve \bb}_K$,
the scalar product
\be
\label{sscalar}
\langle\!\langle v, z\rangle\!\rangle_{H(\new{\breve \bb}_K;K)}:=\langle \partial_{\new{\breve \bb}_K} v,\partial_{\new{\breve \bb}_K} z\rangle_{L_2(K)}+\int_{\partial K_{\!-}} v({\bf s}) z({\bf s}) |({\textstyle \frac{\new{\breve \bb}_K}{|\new{\breve \bb}_K|}}\cdot {\bf n}_K) ({\bf s})| r({\bf s})d{\bf s}.
\ee
gives rise to a (uniform) {\em equivalent} norm $\tvert \cdot \tvert_{H(\new{\breve \bb}_K;K)}$ on $H(\new{\breve \bb}_K;K)$.
\end{proposition}

The corresponding global versions read $\langle\!\langle \cdot, \cdot \rangle\!\rangle_{H(\new{\breve \bb};\tria_s)}=\sum_{K \in \tria_s}\langle\!\langle \cdot|_K, \cdot|_K\rangle\!\rangle_{H(\new{\breve \bb}_K;K)}$, and so
$\tvert \cdot \tvert_{H(\new{\breve \bb};\tria_s)}:=\sqrt{\sum_{K \in \tria_s} \tvert \cdot|_K \tvert_{H(\new{\breve \bb}_K;K)}^2}$.

For the next observation it is convenient to use the shorthand notations
$$
\framebox{$\mu:=w-u$}, \quad
\framebox{$\lambda:= \partial_\new{\tilde \bb} w + \new{\tilde c} w  -{\tilde f}$}, \quad
\framebox{$\gamma:= \lambda-(\partial_\new{\tilde \bb} \mu+\new{\tilde c}\mu+\new{\breve d}_K\mu)$},
$$
so that, in particular,
$$
\gamma =\partial_\new{\tilde \bb} u + \new{\tilde c}u - {\tilde f} - \new{\breve d}_K(w-u).
$$
\new{Note also that $\EE(u,w,\tilde f)^2 \eqsim \sum_{K \in \tria_s} \|\mu\|^2_{L_2(K)}+ \|\lambda\|^2_{L_2(K)}$.}

For smooth $u$, $w$, and $\tilde f$ on $K$, the solution $\breve{R}_K=\breve{R}_K(u,w;f) \in H(\new{\breve \bb}_K;K)$
 of the variational problem
 \begin{equation} \label{locallift}
 \langle\!\langle  \breve{R}_K, v\rangle\!\rangle_{H(\new{\breve \bb}_K;K)}=\breve{b}_K(u,w;v) -\int_K {\tilde f} v\,dx \quad(v \in H(\new{\breve \bb}_K;K)),
 \end{equation}
is the (strong) solution of 
 \begin{equation} \label{18}
\left\{
\begin{array}{rcll}
-\partial_{\new{\breve \bb}_K}^2 \breve{R}_K & = & \gamma & \text{on } K,\\
\partial_{\new{\breve \bb}_K} \breve{R}_K - r |\new{\breve \bb}_K|^{-1} \breve{R}_K & = &  \mu & \text{on } \partial K_{\!-},\\
\partial_{\new{\breve \bb}_K} \breve{R}_K & = &  \mu & \text{on } \partial K_{\!+}.
\end{array}
\right.
\end{equation}

This $\breve{R}_K$ is the exact Riesz lift of the local modified residual $v \mapsto \breve{b}_K(u,w|_{\partial K},v)-\int_K {\tilde f} v\,d{\bf x} \in H(\new{\breve \bb}_K;K)'$, with $H(\new{\breve \bb}_K;K)$ being equipped with $\langle\!\langle \cdot, \cdot\rangle\!\rangle_{H(\new{\breve \bb}_K;K)}$.

To identify next $\breve{R}_K$ exactly, let $(x_1,\ldots,x_n)$ denote Cartesian coordinates on $K$ with the first basis vector being equal to $\new{\breve \bb}_K / |\new{\breve \bb}_K|$.
For ${\bf x}=(x,{\bf y}) \in K$, let $x_{\!\pm}({\bf y})$ be such that $(x_{\!\pm}({\bf y}),{\bf y})  \in \partial K_{\!\pm}$, see Figure~\ref{fig1}. 
\begin{figure}
\begin{center}
  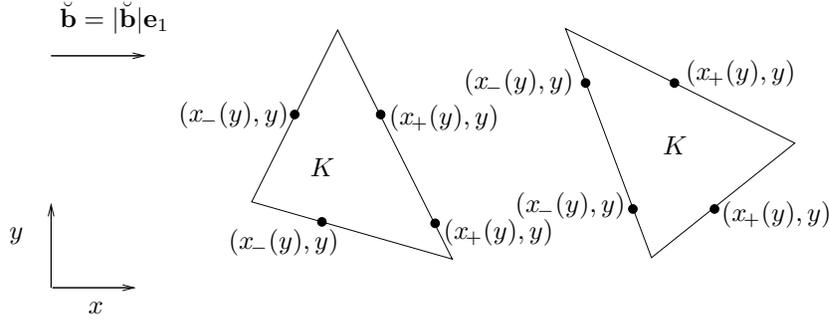
\end{center}
\caption{$x_{\!\pm}$ on a triangle $K$ with two (left) or one (right) inflow boundaries. The enclosing triangle $\bar{K}$ and $\bar{x}_{-}$ will get their meaning in Sect.~\ref{Sapproxliftres}.}
\label{fig1}
\end{figure}

The solution $\breve{R}_K$ reads then as
\begin{equation} 
\label{optimaltest}
\begin{split}
\breve{R}_K(x,{\bf y})=&- |{\new{\breve \bb}_K}|^{-2} \int_{x_{\!-}({\bf y})}^x \int_{x_{\!-}({\bf y})}^z \gamma(q,{\bf y}) d q
dz
\\
 &+ \Big(|{\new{\breve \bb}_K}|^{-1}\mu(x_{\!+}({\bf y}),{\bf y})+|{\new{\breve \bb}_K}|^{-2} \int_{x_{\!-}({\bf y})}^{x_{\!+}({\bf y})} \gamma(q,{\bf y}) dq \Big)\Big(x-x_{\!-}({\bf y})\Big)\\
&+ \frac{ \int_{x_{\!-}({\bf y})}^{x_{\!+}({\bf y})} (\partial_{\new{\breve \bb}_K} \mu+\gamma) (q,{\bf y}) dq}{x_{\!+}({\bf y})-x_{\!-}({\bf y})},
\end{split}
\end{equation}
and is seen to be piecewise polynomial over $K$ when $\gamma,\mu$ are polynomial over $K$.

\subsection{Approximate lifted residuals} \label{Sapproxliftres}
Next we define an approximation $\brbrR_K$ to $\breve{R}_K$ by discarding higher order terms. Whereas, for polynomial $u$, $w$, ${\tilde f}$, $\new{\tilde \bb}$, and $\new{\tilde c}$ on $K$, $\breve{R}_K$ is only {\em piecewise} polynomial w.r.t. a partition of $K$ into subsimplices (indicated by the dotted lines in Figure~\ref{fig1}) that depends on the field $\new{\breve \bb}_K$, $\brbrR_K$ will always be {\em polynomial} on $K$.
 
The reason for introducing $\brbrR_K$ is that $\new{\breve \bb}^\perp \cdot \nabla \breve{R}_K$ can be arbitrarily large, which would not allow us to perform Step~{\bf (II)} on page~\pageref{steps} of our proof. 
This is caused by the fact that the subdivision of $K$ into  the aforementioned subsimplices can have arbitrarily small angles, and thus impedes a useful application of the inverse (or Bernstein) inequality to $\breve{R}_K$.

To define $\brbrR_K$, first we construct a {\em polyhedral} set  $\bar{K}$ that contains $K$ as follows. The number of inflow faces of $K$ is between $1$ and $n-1$ where $n$ is the spatial dimension. Let $F$ be the inflow face whose normal makes the smallest angle with  ${\new{\breve \bb}_K}$, and let $v$ denote the vertex of $K$ that does not belong to $F$. Finally let $H_F$ denote the $(n-1)$-hyperplane containing $F$. The ``shadow'' of $K$ on $H_F$, i.e.,
$$
\bar F := \big\{\bx\in H_F: \{\bx + t{\new{\breve \bb}_K}: t\in \R\} \cap K \neq \emptyset\big\},
$$
is an $(n-1)$-dimensional polyhedron containing $F$. Let $\bar K$ denote the convex hull  of $v$ and $\bar F$,  cf. Figure~\ref{fig1} for $n=2$.
Then, by construction, $\bar{K}$ has only one inflow face $\new{\partial}\bar{K}_{-} := \bar F$, and 
$K \subseteq \bar{K}$ with equality if and only if $K$ has only {\em one} inflow face, namely $\new{\partial}K_- =F$.

For ${\bf x}=(x,{\bf y}) \in \bar{K} \supset K$, let ${\bf x} \mapsto \bar{x}_{\!-}({\bf y}) \in \cP_1(K)$ be the linear function with $(\bar{x}_{\!-}({\bf y}),{\bf y})  \in \partial \bar{K}_{-}$,
i.e., $\bar{x}_{\!-}(\by)$ agrees with $x_{\!-}(\by)$ on $F$.
Then we have
\begin{align}
&\diam (\bar{K}) \lesssim \diam(K),\label{108}\\
&|\bar{x}_{\!-}|_{W^1_\infty(\bar{K})} \lesssim 1, \label{104}
\end{align}
where both constants  depend only on (an upper bound for) $\varrho_K$.

We define the {\em approximate lifted local residual}
$\brbrR_K=\brbrR_K(u,w;f) \in \cP_{m_v}(K)$ (cf. \eqref{mv})
by
\begin{equation} \label{near-opt}
\brbrR_K(x,{\bf y}):= |{\new{\breve \bb}_K}|^{-1} \mu(\bar{x}_{\!-}({\bf y}),{\bf y}) (x-\bar{x}_{\!-}({\bf y}))+\big(\lambda-(\new{\tilde c}+\new{\breve d}_K)\mu\big)
(\bar{x}_{\!-}({\bf y}),{\bf y}).
\end{equation}
Note that $\partial_{\new{\breve \bb}_K} \brbrR_K=\mu(\bar{x}_{\!-}({\bf y}),{\bf y})$.

The following lemmas show how $\brbrR_K$ relates on the one hand  to the exact Riesz lift $\breve{R}_K$
and on the other hand to the ``residuals'' $\mu=w-u$, $\lambda=\partial_\new{\tilde \bb} w+\new{\tilde c}w-\new{\tilde f}$ on $K$.
\begin{lemma} 
\label{lem3} For $\diam(K) \leq |{\new{\breve \bb}_K}|$, it holds that
$$
\|\breve{R}_K-\brbrR_K\|_{H(\new{\breve \bb}_K;K)} \lesssim  |\new{\breve \bb}_K|^{-1} \diam(K) \big(\| \mu\|_{H(\new{\breve \bb}_K;\bar{K})}+\|\lambda\|_{H(\new{\breve \bb}_K;\bar{K})}\big)+\diam(K)\|\mu\|_{H^1(K)},
$$
with a constant depending only on (upper bounds for) $\|c\|_{\new{W^1_\infty(K)}}$, \new{$|\bb|_{W^1_\infty(K)^n}$}, and $\varrho_K$.
\end{lemma}

\begin{proof}
We write $\breve{R}_K-\brbrR_K$ as 
\begin{align} \label{line1}
&|{\new{\breve \bb}_K}|^{-2} \Big(
\big(x-x_{\!-}({\bf y})\big) \int_{x_{\!-}({\bf y})}^{x_{\!+}({\bf y})} \gamma(q,{\bf y}) dq -
\int_{x_{\!-}({\bf y})}^x \int_{x_{\!-}({\bf y})}^z \gamma(q,{\bf y}) d q
 \Big)+\\ \label{line2}
& |{\new{\breve \bb}_K}|^{-1} \big(\mu(x_{\!+}({\bf y}),{\bf y}) (x-x_{\!-}({\bf y}))- \mu(\bar{x}_{\!-}({\bf y}),{\bf y})(x-\bar{x}_{\!-}({\bf y}))\big)+\\ \label{line3}
& \frac{ \int_{x_{\!-}({\bf y})}^{x_{\!+}({\bf y})} (\partial_{\new{\breve \bb}_K} \mu+\gamma) (q,{\bf y}) dq}{x_{\!+}({\bf y})-x_{\!-}({\bf y})}-\big(\lambda-(\new{\tilde c}+\new{\breve d}_K)\mu\big)
(\bar{x}_{\!-}({\bf y}),{\bf y}).
 \end{align}
 
 Writing $\mu(x_{\!+}({\bf y}),{\bf y})$ as  $\mu(x,{\bf y})+|\new{\breve \bb}_K|^{-1} \int^{x_{\!+}({\bf y})}_{x}\partial_{\new{\breve \bb}_K} \mu(q,{\bf y})dq$, and similarly for $\mu(\bar{x}_{\!-}({\bf y}),{\bf y})$, and using that $\diam(K) \leq |{\new{\breve \bb}_K}|$, one infers that the $L_2(K)$-norm of \eqref{line2} is 
 \begin{align*}
 & \lesssim 
 |\new{\breve \bb}_K|^{-1} \diam(\bar{K}) \|\mu\|_{L_2(K)}+|\new{\breve \bb}_K|^{-2} \diam(\bar{K})^2\|\partial_{\new{\breve \bb}_K} \mu\|_{L_2(\bar{K})}\\
  & \lesssim |\new{\breve \bb}_K|^{-1} \diam(K)\|\mu\|_{H(\new{\breve \bb}_K;\bar{K})},
 \end{align*}
with a constant  only depending on $\varrho_K$.
 
The $L_2(K)$-norm of \eqref{line1} in turn  is 
\begin{align*}
 & \lesssim 
 |\new{\breve \bb}_K|^{-2} \diam(K)^2 \|\gamma\|_{L_2(K)}\\
 & \leq |\new{\breve \bb}_K|^{-2} \diam(K)^2 \big(\|\lambda\|_{L_2(K)}+(\|\new{\tilde c}\|_{L_\infty(K)} + |\new{\breve d}_K|)\|\mu\|_{L_2(K)}+\||\new{\tilde \bb}|\|_{L_\infty(K)}|\mu|_{H^1(K)}\big)\\
& \lesssim |\new{\breve \bb}_K|^{-1} \diam(K) \big(\|\lambda\|_{L_2(K)}+ \|\mu\|_{L_2(K)}\big)+\diam(K)|\mu|_{H^1(K)},
 \end{align*}
with a constant depending only   on (upper bounds for) $\new{\|\tilde c\|_{L_\infty(K)} \lesssim} \|c\|_{L_\infty(K)}$, and \new{$\|\divv \tilde \bb\|_{L_2(K)} \lesssim |\tilde \bb|_{W^1_\infty(K)^n} \lesssim  |\bb|_{W^1_\infty(K)^n} $}, where we have used that 
$\||\new{\tilde \bb}-\new{\breve \bb}_K|\|_{L_\infty(K)} \lesssim \diam(K)|\new{\tilde \bb}|_{W^1_\infty(K)^n}$ and $\diam(K) \leq |{\new{\breve \bb}_K}|$.

Using that $\partial_{\new{\breve \bb}_K} \mu+\gamma=\lambda-(\new{\tilde c}+\new{\breve d}_K)\mu+(\new{\breve \bb}_K-\new{\tilde \bb})\cdot \nabla \mu$,
 we find that the $L_2(K)$-norm of \eqref{line3} is bounded by a constant multiple of
\begin{align*}
& 
|\new{\breve \bb}_K|^{-1}\diam(\bar{K}) \|\partial_{\new{\breve \bb}_K}(\lambda-(\new{\tilde c}+\new{\breve d}_K)\mu)\|_{L_2(\bar{K})}+|\new{\tilde \bb}|_{W^1_\infty(K)^n} \diam(K) |\mu|_{H^1(K)}\\
& \lesssim 
|\new{\breve \bb}_K|^{-1}\diam(K)\big(\|\partial_{\new{\breve \bb}_K}\lambda\|_{L_2(\bar{K})}+\|\mu\|_{\new{H({\breve \bb}_K;\bar{K})}}\big)+\diam(K)\|\mu\|_{H^1(K)},
\end{align*}
only dependent on (upper bounds for) $\new{\|{\tilde c}\|_{W^1_\infty(K)} \lesssim} \|c\|_{W^1_\infty(K)}$, $|\bb|_{W^1_\infty(K)^n}$, and $\varrho_K$.

Next, we write
$$
\partial_{\new{\breve \bb}_K} \big(\breve{R}_K(x,{\bf y})-\brbrR_K(x,{\bf y})\big)=
 \mu(x_{\!+}({\bf y}),{\bf y})+|{\new{\breve \bb}_K}|^{-1} \int_{x}^{x_{\!+}({\bf y})} \gamma(q,{\bf y}) dq -\mu(\bar{x}_{\!-}({\bf y}),{\bf y}).
$$
 Its $L_2(K)$-norm is 
 \begin{align*}
& \lesssim |\new{\breve \bb}_K|^{-1} \diam(K) \big(\|\partial_{\new{\breve \bb}_K} \mu\|_{L_2(\bar{K})}+\|\gamma\|_{L_2(K)}\big)
\\ & \lesssim |\new{\breve \bb}_K|^{-1} \diam(K) \big(\| \mu\|_{H(\new{\breve \bb}_K;\bar{K})}+\|\lambda\|_{L_2(K)}\big),
\end{align*}
 only dependent on (upper bounds for) $\|c\|_{L_\infty(K)}$, \new{$\|\bb \|_{W_\infty^1(K)^n}$}, and $\varrho_K$.
 By collecting the derived upper bounds, the proof is completed.
 \end{proof}

 \new{We end this subsection with another technical lemma which will play a key role to prove Step~{\bf (I)} on page~\pageref{steps}.
In fact, using that $\lambda$ and $\mu$ are piecewise polynomial on $\tria$,  inverse inequalities will allow us to show that
 the terms involving first order derivatives can be kept small relative to the other ones by choosing the subgrid depth sufficiently large.
 Then, the next lemma in conjunction with the previous Lemma \ref{lem3} already hints at the fact that $\EE(u,w,\tilde f)$ provides a lower bound for
 $\|\breve{R}\|_{H(\breve{\bb};\tria_s)}$. It then remains to switch to the correct norm to establish Step (I), see Corollary \ref{corol1} below.}

 \begin{lemma} 
\label{lem6} 
 {For $\diam(K) \leq |\new{\breve \bb}_K|$, it holds that
\be \label{201a}
 \|\brbrR_K\|^2_{H({\new{\breve \bb}_K};K)} + \diam(K)^2 (|\mu|_{H^1(\bar{K})}^2+|\lambda|_{H^1(\bar{K})}^2) \gtrsim \|\lambda\|_{L_2(K)}^2+\|\mu\|_{L_2(K)}^2,
 \ee
where the constant depends only   on (upper bounds for) $\|c\|_{L_\infty(\Omega)}$, $|\new{\tilde \bb}|_{W^1_\infty(\Omega)^n}$, and  $\varrho_K$.}
 \end{lemma}

 \begin{proof} By $\diam(\bar{K}) \lesssim \diam(K) \leq |\new{\breve \bb}_K|$, similarly as in the proof of Lemma~\ref{lem3}, one infers that
 $\|\brbrR_K-\lambda\|_{L_2(K)} \lesssim \|\mu\|_{L_2(K)}+\diam(K)(|\mu|_{H^1(\bar{K})} +|\lambda|_{H^1(\bar{K})})$
 and $\|\partial_{\new{\breve \bb}_K} \brbrR_K-\mu\|_{L_2(K)} \lesssim \diam(K)|\mu|_{H^1(\bar{K})}$, with constants depending only  on (upper bounds for) $\|c\|_{L_\infty(\Omega)}$, $|\new{\tilde \bb}|_{W^1_\infty(\Omega)^n}$,  and  $\varrho_K$.
 
 By two applications of Young's inequality, we infer that
 \begin{align*}
  \|\brbrR_K\|^2_{L_2(K)} + \|\partial_{\new{\breve \bb}_K} \brbrR_K\|^2_{L_2(K)}  \geq & (1-\eta) \|\lambda\|^2_{L_2(K)}-(\eta^{-1}-1)\|\brbrR_K-\lambda\|^2_{L_2(K)}\\
  & +
 {\textstyle \frac{1}{2}} \|\mu\|^2_{L_2(K)}-(2-1)\|\partial_{\new{\breve \bb}_K} \brbrR_K-\mu\|^2_{L_2(K)} .
 \end{align*}
 By selecting the constant $\eta \in (0,1)$ sufficiently close to $1$, the proof is completed.
 \end{proof}

\subsection{Proof of Proposition~\ref{lemma1new}} \label{Sproof}
So far we have not used that $u$, $w$, $\new{\tilde f}$, $\new{\tilde \bb}$, and $\new{\tilde c}$ are piecewise polynomial w.r.t. $\tria$, whilst $\new{\breve b}_{\tria_s}(\,,\,)$ is a `broken' bilinear form 
w.r.t. a \new{sufficiently} refined partition $\tria_s$, \new{and furthermore that $|\bb|^{-1} \in L_\infty(\Omega)$}. These facts are going to be used in the following.

Setting
$$
D:=\sup_{K \in \mathfrak{T}} \sup_{0 \neq {\bf b} \in W^1_\infty(K)^{\new{n}}} \frac{\| |\new{\breve \bb}_K-{\bf b}|\|_{L_\infty(K)}}{\diam(K)  |{\bf b}|_{W_\infty^1(K)^n}}\quad(<\infty),
$$
 let
\begin{equation} \label{111}
\bar{\sigma}>0
\end{equation}
be such that for $\sigma \in (0,\bar{\sigma}]$ and $\tria \in \T$, $\tria_s$ is sufficiently fine to ensure that
\begin{equation} \label{100}
 \diam(K)\, \||{\bf b}|^{-1}\|_{L_\infty(K)} \max\big(1,D |{\bf b}|_{W_\infty^1(K)^n}\big) \leq {\textstyle \frac{1}{2}} \quad (K \in \tria_s).
\end{equation}
 Then for any $K \in \tria_s$, we have
\begin{equation} \label{102}
\begin{split}
|\new{\breve \bb}_K| &\geq \||{\bf b}|^{-1}\|_{L_\infty(K)}^{-1}-\||\new{\breve \bb}_{K}-{\bf b}|\|_{L_\infty(K)}  \\
&\geq \||{\bf b}|^{-1}\|_{L_\infty(K)}^{-1} -D \diam(K) |{\bf b}|_{W_\infty^1(K)^n}\\
&\geq {\textstyle \frac{1}{2}}  \||{\bf b}|^{-1}\|_{L_\infty(K)}^{-1} \geq \max\big({\textstyle \frac{1}{2}}  \||{\bf b}|^{-1}\|_{L_\infty(\Omega)}^{-1},\diam(K)
\big),
\end{split}
\end{equation}
where we have used \eqref{100}.

For $K \in\mathfrak{T}$, and $k \geq \ell \in \N_0$, we will make repeated use of the {\em inverse inequality}
$$
|\cdot|_{H^k(K)} \lesssim \diam(K)^{-(k-\ell)}\|\cdot\|_{H^\ell(K)}\quad \text{on } \cP_m(K),
$$
where the constant depends only  on $m$, $\varrho_K$, and $k$.\\

\begin{corollary} \label{corol1}
\new{We define $\breve{R}$, $\brbrR$ by $\breve{R}|_K:=\breve{R}_K$ and $\brbrR|_K:=\brbrR_K$ for $K \in \tria_s$. Then, one
has for }
  ${(u,w,f) \in \DD_\tria}$, $\sigma \in (0,\bar{\sigma}]$   that
\renewcommand{\theenumi}{\roman{enumi}}
\begin{enumerate}
\item \label{i} $\|\breve{R}-\brbrR\|_{H({\breve \bb};\tria_s)} \lesssim \sigma  {\EE(u,w,\tilde f)}$,
\item \label{ii}  {$\|\brbrR\|_{H({\breve \bb};\tria_s)}  \gtrsim {\EE(u,w,\tilde f)}$}, \new{provided that 
 $\sigma \in (0,\sigma_0]$ with $\sigma_0 \in (0,\bar{\sigma}]$ being sufficiently small.}
\end{enumerate}
Both  constants hidden in the $\lesssim$ and $\gtrsim$ symbols, and the upper bound for $\sigma_0$ depend only on the quantities mentioned in the statement of Proposition~\ref{lemma1new}.
\end{corollary}

\begin{proof} For $K' \in \tria$ and $p \in \cP_m(K')$, we have that
\begin{equation} \label{213}
\sum_{\{K \in \tria_s\colon K \subset K'\}}  |p|^2_{H^1(\bar{K})} \eqsim 
\sum_{\{K \in \tria_s\colon K \subset K'\}}  |p|^2_{H^1(K)}
= |p|^2_{H^1(K')} \lesssim
\diam(K')^{-2} \|p\|^2_{L_2(K')},
\end{equation}
with a constant depending on $\varrho$ and $m$.

By applying this type of estimate to $\lambda$ and $\mu$, preceded by an application of  Lemma~\ref{lem3} \new{whilst using $|\breve{\bb}_K|^{-1} \leq 2 \||\bb|^{-1}\|_{L_\infty(\Omega)}$ and $|\breve{\bb}_K|\leq \|\tilde \bb\|_{L_\infty(K)} \leq 2 \|\bb\|_{L_\infty(\Omega)}$}, we obtain
$$
\|\breve{R}-\brbrR\|_{H(\new{\breve \bb};\tria_s)} \lesssim \sigma \new{\EE(u,w,\tilde f)}.
$$

By summing the result of Lemma~\ref{lem6} over $K \in \tria_s$ and applying \eqref{213} with $p=\mu$ and $p=\lambda$, 
we infer that for $\sigma$ small enough,
$\|\brbrR\|_{H(\new{\breve \bb};\tria_s)} \gtrsim \new{\EE(u,w,\tilde f)}$. 
\end{proof}

The next proposition is almost Step~{\bf (I)} on page~\pageref{steps}, except that we still have to replace $\|\brbrR\|_{H(\new{\breve \bb};\Omega)}$ by $\|\brbrR\|_{\V_{\tria_s}}$, which will be done using the subsequent Lemma~\ref{lem4}\eqref{b}.

\begin{proposition} \label{prop1}
There exist a $\kappa>0$ and a $\sigma_1 \in (0,\sigma_0]$, 
that depend only on the quantities mentioned in the statement of Proposition~\ref{lemma1new},
 such that for $\sigma \in (0,\sigma_1]$, and any $\new{(u,w,f) \in \DD_\tria}$,
$$
\breve{b}_{\tria_s}(u,w;\brbrR)-\int_\Omega f \brbrR\,d{\bf x} \geq \kappa \new{\EE(u,w,\tilde f)} \|\brbrR\|_{H(\new{\breve \bb};\tria_s)}.
$$
\end{proposition}

\begin{proof} With $\breve{R}|_K:=\breve{R}_K(u,w;f)$, its definition in \eqref{locallift} shows that
$$
\breve{b}_{\tria_s}(u,w;\brbrR)-\int_\Omega f \brbrR\,d{\bf x}=\sum_{K \in \tria_s}  \langle\!\langle   {\breve{R}}_K,  \brbrR_K\rangle\!\rangle_{H(\new{\breve \bb}_K;K)}.
$$
Thanks to the equivalence of norms from Proposition~\ref{equivalence}, an application of  Corollary~\ref{corol1}\eqref{i} shows that
$$
\big| \sum_{K \in \tria_s} \langle\!\langle   \breve{R}_K-\brbrR_K,  \brbrR_K\rangle\!\rangle_{H(\new{\breve \bb}_K;K)}\big| \lesssim  \sigma 
\new{\EE(u,w,\tilde f)}
 \|\brbrR\|_{H(\new{\breve \bb};\tria_s)}.
$$
For $\sigma$ being sufficiently small, an application of Corollary~\ref{corol1}\eqref{ii} shows that
$$
\tvert \brbrR \tvert_{H(\new{\breve \bb};\tria_s)}^2
 \eqsim\|\brbrR\|^2_{H(\new{\breve \bb};\tria_s)}
\gtrsim \new{\EE(u,w,\tilde f)} \|\brbrR\|_{H(\new{\breve \bb};\tria_s)},
$$
by which the proof is easily completed.
\end{proof}

\begin{lemma} \label{lem4} For $\new{(u,w,f) \in \DD_\tria}$, $\sigma \in (0,\sigma_0]$, it holds that
\renewcommand{\theenumi}{\alph{enumi}}
\begin{enumerate}
\item \label{a} $\sum_{K \in \tria_s} \diam(K)^2\|\brbrR_K\|_{H^1(K)}^2 \lesssim \sigma^2 \|\brbrR\|_{H(\new{\breve \bb};\Omega)}^2$,
\item \label{b} $\big| \|\brbrR\|_{H(\new{\breve \bb};\tria_s)} - \|\brbrR\|_{\V_{\tria_s}}\big| \lesssim \sigma  \|\brbrR\|_{H(\new{\breve \bb};\tria_s)}$,
\end{enumerate}
depending only on the quantities mentioned in the statement of Proposition~\ref{lemma1new}.
\end{lemma}

\begin{proof}\eqref{a}.
For $K \in \tria_s$, we split $\brbrR_K=\brbrR_{K,1}+\brbrR_{K,2}+\brbrR_{K,3}$ defined by
\begin{align*}
\brbrR_{1,K}(x,{\bf y})& :=|{\new{\breve \bb}_K}|^{-1} \mu(\bar{x}_{\!-}({\bf y}),{\bf y}) (x-\bar{x}_{\!-}({\bf y})),\\
\brbrR_{2,K}(x,{\bf y})& :=\big(\lambda-(\new{\tilde c}+\new{\breve d}_{K'})\mu \Big)(\bar{x}_{\!-}({\bf y}),{\bf y}),\\
\brbrR_{3,K}(x,{\bf y})& :=(\new{\breve d}_{K'}-\new{\breve d}_K)\mu (\bar{x}_{\!-}({\bf y}),{\bf y}).
\end{align*}
where $K' \in \tria$ is such that $K \subset K'$. Correspondingly, we split $\brbrR=\brbrR_1+\brbrR_2+\brbrR_3$.

Since $\brbrR_{1,K}$ vanishes on $\partial \bar{K}_{\!-}$, an application of Poincar\'{e}'s inequality on each  streamline following $\new{\breve \bb}_K$ shows that
$ \|\brbrR_{1,K}\|_{L_2(\bar{K})} \lesssim |\new{\breve \bb}_K|^{-1} \diam(\bar{K})  \|\partial_{\new{\breve \bb}_K} \brbrR_{1,K}\|_{L_2(\bar{K})}$ (cf. possibly \cite[Prop.~4.3]{35.8566}).
From the fact that $\brbrR_{1,K}$ is polynomial, $\diam(\bar{K}) \lesssim \diam(K)$,
and $\partial_{\new{\breve \bb}_K} \brbrR_{1,K}=\partial_{\new{\breve \bb}_K} \brbrR_{K}$,
 by an application of the inverse inequality we obtain $\sum_{K \in \tria_s} \diam(K)^2\|\brbrR_{1,K}\|_{H^1(K)}^2 \lesssim \sum_{K \in \tria_s} |\new{\breve \bb}_K|^{-2}\diam(K)^2
 \|\partial_{\new{\breve \bb}_K} \brbrR_K\|_{L_2(K)}^2 \lesssim \sigma^2 \|\brbrR\|_{H(\new{\breve \bb};\Omega)}^2$.

Recalling from \eqref{104} that  $|\bar{x}_{\!-}|_{W^1_\infty(K)} \lesssim 1$, and since $\lambda-(\new{\tilde c}+\new{\breve d}_{K'})\mu$ is polynomial on $K$, we have $\|(x,{\bf y})\mapsto (\lambda-(\new{\tilde c}+\new{\breve d}_{K'})\mu)(\bar{x}_{\!-}({\bf y}),{\bf y})\|_{H^1(K)} \lesssim \|\lambda-(\new{\tilde c}+\new{\breve d}_{K'})\mu\|_{H^1(K)}$.
Now using that for $\tria \ni K' \supset K$, $\lambda-(\new{\tilde c}+\new{\breve d}_{K'})\mu$ is polynomial on $K'$, an application of \eqref{213} shows that
$$
\sum_{K \in \tria_s} \diam(K)^2\|\brbrR_{2,K}\|_{H^1(K)}^2 \lesssim
 \sigma^2 \new{\EE(u,w,\tilde f)^2}
\lesssim \sigma^2 \|\brbrR\|_{H(\new{\breve \bb};\tria_s)}^2,
$$
where the last inequality follows from Corollary~\ref{corol1}\eqref{ii}.

Again by $|\bar{x}_{\!-}|_{W^1_\infty(K)} \lesssim 1$, we have \new{$\|\brbrR_{3,K}\|_{H^1(K)} \lesssim |\new{\tilde \bb}|_{W^1_\infty(K')} \|\mu\|_{H^1(K)}$,} which together with 
 \eqref{213} yields that
$$
\sum_{K \in \tria_s} \diam(K)^2 \|\brbrR_{3,K}\|^2_{H^1(K)} \lesssim \sum_{K \in \tria_s} \diam(K)^2 \|\mu\|_{L_2(K)}^2 
\lesssim \sigma^2 \|\brbrR\|_{H(\new{\breve \bb};\tria_s)}^2.
$$ 
by Corollary~\ref{corol1}, \eqref{ii}, which completes the proof of \eqref{a}.

\eqref{b}. From the triangle inequality, and $\||\new{\breve \bb}_K-\bb|\|_{L_\infty(K)} \leq D \diam(K) |\bb|_{W^1_\infty(K)^n}$, we infer that
$\big| \|\brbrR\|_{H(\new{\breve \bb};\tria_s)} - \|\brbrR\|_{\V_{\tria_s}}\big| \lesssim \sqrt{\sum_{K \in \tria_s} \diam(K)^2|\brbrR_K|_{H^1(K)}^2} \lesssim \sigma \|\brbrR\|_{H(\new{\breve \bb};\Omega)}$ by \eqref{a}, which is \eqref{b}.
\end{proof}

Proposition~\ref{prop1} together with Lemma~\ref{lem4}\eqref{b} complete the proof of 
$$
\breve{b}_{\tria_s}(u,w;\brbrR)-\int_\Omega f \brbrR\,d{\bf x} \geq \kappa \new{\EE(u,w,\tilde f)} \|\brbrR\|_{\V_{\tria_s}}
$$
for sufficiently small $\sigma>0$, being Step~{\bf (I)} in our proof of Proposition~\ref{lemma1new}.

Step~{\bf (II)} is implied by the next result when we use that $\|u-w\|_{L_2(\Omega)} \leq \new{\EE(u,w,\tilde f)}$.

\begin{proposition} For $\new{(u,w,f) \in \DD_\tria}$, $\sigma \in (0,\sigma_0]$ sufficiently small, it holds that
$$
|\new{\tilde b}_{\tria_s}(u,w;\brbrR) - \breve b_{\tria_s}(u,w;\brbrR)|  \lesssim  \sigma \|u-w\|_{L_2(\Omega)}\|\brbrR\|_{\V_{\tria_s}}.
$$
Both the upper bound for $\sigma$ and the constant hidden in the $\lesssim$-symbol 
depend only on the quantities mentioned in the statement of Proposition~\ref{lemma1new}.
\end{proposition}

\begin{proof} For $K \in \tria_s$ and sufficiently smooth $u$, $w$, and $v$, it holds that
$$
\new{\tilde b}_K(u,w;v)=\int_K (\new{\tilde c}u +\new{\tilde \bb}\cdot \nabla u)v \,d{\bf x}+\int_{\partial K} \new{\tilde \bb}\cdot {\bf n}_K (w-u)v \,ds,
$$
and so
$$
\new{\tilde b}_K(u,w;v) - \breve b_K(u,w;v)=-\int_K \new{\breve d}_K (w-u)v\,d{\bf x}+\int_{\partial K} (\new{\tilde \bb}-\new{\breve \bb}_K)\cdot {\bf n}_K(w-u)v \,ds.
$$
With $z:=(w-u)v$, and $\bar{z}:=|K|^{-1} \int_K z\,d{\bf x}$, recalling that $\new{\breve d}_K=|K|^{-1}\int_K \divv \new{\tilde \bb}\,d {\bf x}$ an application of the trace theorem shows that
\begin{align*}
\Big|-\int_K \new{\breve d}_K z\,d{\bf x}&+\int_{\partial K} (\new{\tilde \bb}-\new{\breve \bb}_K)\cdot {\bf n}_Kz \,ds\Big|\\
&=\Big|-\int_K \new{\breve d}_K (z-\bar{z}) \,d{\bf x}+\int_{\partial K} (\new{\tilde \bb}-\new{\breve \bb}_K)\cdot {\bf n}_K (z-\bar{z}) \,ds\Big|\\
&\lesssim \|\divv \new{\tilde \bb}\|_{L_\infty(K)} \|z-\bar{z}\|_{L_1(K)}+\||\new{\tilde \bb}-\new{\breve \bb}_K|\|_{L_\infty(K)} |z|_{W^1_1(K)}\\
&\lesssim \diam(K)|z|_{W^1_1(K)}\\
& \lesssim \diam(K)\big(\|w-u\|_{L_2(K)}|v|_{H^1(K)}+\|v\|_{L_2(K)} |w-u|_{H^1(K)}\big).
\end{align*}
By substituting $v=\brbrR_K$, summing over $K \in \tria_s$, and applying the Cauchy-Schwarz inequality we find that
\begin{align*}
|\new{\tilde b}_{\tria_s}(u,w;\brbrR) - \breve b_{\tria_s}(u,w;\brbrR)|
& \lesssim  \|w-u\|_{L_2(\Omega)} \sqrt{\sum_{K \in \tria_s} \diam(K)^2|\brbrR|_{H^1(K)}^2} \\ & \quad+\|\brbrR\|_{L_2(\Omega)} 
\sqrt{\sum_{K \in \tria_s} \diam(K)^2|u-w|_{H^1(K)}^2}\\
& \lesssim \sigma \|u-w\|_{L_2(\Omega)}\|\brbrR\|_{H(\new{\breve \bb};\tria_s)}
\end{align*}
where we have applied \eqref{213} and Lemma~\ref{lem4}\eqref{a}. Finally, for sufficiently small $\sigma$, in the last expression  $\|\brbrR\|_{H(\new{\breve \bb};\tria_s)}$ can be replaced in view of Lemma~\ref{lem4}\eqref{b}  by $\|\brbrR\|_{\V_{\tria_s}}$.
\end{proof}

Since we have performed Steps~{\bf (I)}-{\bf (II)} on page~\pageref{steps}, the proof of Proposition~\ref{lemma1new} is complete.\hfill $\Box$
\section{Effective  Mark and Refinement Strategy for an Adaptive DPG  method}\label{sec:red}
The key common ingredient of an adaptive solution strategy for a PDE is a collection of local error indicators associated
with the current partition $\Tr$ underlying the discretization. While an individual indicator does not characterize the actual local error
the accumulation of all indicators is equivalent to the global current approximation error.
Based on the error indicators one contrives a {\em marking strategy}
which identifies a subset $\cM \subset \Tr$ of {\em marked cells} to be refined in the subsequent adaptive step.
The perhaps most prominent marking strategy is based on a {\em bulk criterion}, sometimes called "D\"orfler Marking" where
one collects (a possibly small number of) cells for which the accumulated combined indicators capture at least a given fixed portion of
the global a posteriori error bound. While this is usually perceived as a heuristically very plausible strategy, a rigorous convergence and complexity analysis
is actually quite intricate. It typically comes in two stages, namely establishing first that such a strategy reduces the current error by a fixed 
ratio, and second to estimate the number of new degrees of freedom incurred by the refinement step.
This paradigm has been studied extensively and is by now well understood  for problems of elliptic type where the dominating effect is diffusion.
The first step of error reduction hinges on (near-){\em Galerkin orthogonality} and is greatly helped by the fact that the \new{common residual based} error indicators contain as an explicit factor a power of the  respective cell diameter. Thus, a refinement does decrease the indicators. 

In the current scenario of transport equations the situation looks similar at  the first glance. Using
$(u, w) \in \breve{\U}^\delta_\Tr$ as primal unknowns, we have a hierarchy of {\em nested} trial spaces at hand, see Remark 
\ref{rem:nested}. Due to the product structure of the test search spaces
we have computable local error indicators associated with the current discretization whose sum is, thanks to Theorem \ref{th2new}, \new{modulo data oscillation} uniformly equivalent to the error in the trial metric.  This suggests using a similar bulk criterion in a {\em mark-and-refine framework} to drive adaptive refinements which is, in fact our choice in the subsequent
discussion.

A closer look reveals, however, some essential distinctions which may actually \new{nourish} some doubts about whether such strategies 
work in a transport problem just as well as in a diffusion problem. The error indicators in the form of projected lifted residuals depend
of course on the mesh defining the DPG scheme but they do {\em not} contain any local mesh size factor that ensures a decay under
refinement. In contrast to the usual way of analyzing residual based a posteriori error estimators
we are able to deduce a fixed error reduction
rate only when starting from a Petrov-Galerkin solution using what one may call {\em Petrov-Galerkin orthogonality} in place of Galerkin orthogonality.
Moreover, there is actually an {\em infinite} family of equivalent a posteriori bounds obtained for {\em any} refinement of the current 
partition arising from different mesh-dependent Riesz liftings. 
A key observation, which we heavily exploit and which may actually be of interest in its own right, is the interrelation of these
error indicators with yet another completely {\em mesh-independent} variant representing the residual for a least squares formulation. 

As indicated by these comments the crucial issue for adaptivity in the context of transport equations is the {\em effectivity} of a given
 mark and refinement strategy in the sense of a guaranteed error reduction rate. The basic structure of a subsequent complexity analysis can instead be expected to be less problem specific. We therefore confine the subsequent discussion entirely to the issue of effectivity which we are currently only able
to  fully establish in one spatial dimension $n=1$.

 For $n>1$ we will employ a {\em downstream enriched} refinement strategy where additionally cells downwind from the marked cells are refined as well.
Our derivation of effectivity in this case will partly be based on a conjecture.

\subsection{Setting and results}\label{ssec:5.1} 
\newcommand{\cD}{\mathcal{D}}
\newcommand{\cS}{\mathcal{S}}
 In view of the already considerable level of technicality we confine the subsequent discussion to the 
 case of a {\em constant }convection field $\bb$, and a piecewise constant reaction coefficient $c$ with respect to the current
 partition $\Tr$  for the trial space.
 In an adaptive setting the latter means that necessarily $c$ is piecewise constant w.r.t. the initial partition $\tria_\bot$,
  i.e., we always assume that
\be
\label{specified}
 \bb(x)\equiv \bb,\quad c=(c_{K'})_{K' \in \tria_{\bot}} \in \cP_0(\tria_{\bot}).
\ee

Given $(u,w) \in \breve{\U}_\tria^\delta$ and $f \in \F_\tria^\delta$, from \eqref{s4}
recall the definition of the {\em projected lifted residual}
$$
R^\delta_{\tria_s}=R^\delta_{\tria_s}(u,w;f)=(R^\delta_K)_{K \in \tria_s} \in \bar{\V}^\delta_{\tria_s} \subset \V_{\tria_s}.
$$
For a collection of marked cells ${\mathcal M} \subset \tria$, we set
$$
\tria_s({\mathcal M})=\{K \in \tria_s\colon K \subset \cup_{K' \in {\mathcal M}} K'\}
$$
for the corresponding portion of the test-subgrid with the convention $\tria_s=\tria_s(\tria)$.
We use the notation $R^\delta_{\tria_s({\mathcal M})}$ to denote $(R^\delta_K)_{K \in \tria_s({\mathcal M})}$.

Aside from a partition $\tria$ and its refinement $\tria_s$, we consider a refined partition $\tilde{\tria}$ with companion refinement $\tilde{\tria}_s$.
Note that $\breve{\U}^\delta_\tria \subset \breve{\U}^\delta_{\tilde{\tria}}$, $\bar{\V}^\delta_{\tria_s} \subset \bar{\V}^\delta_{\tilde{\tria}_s}$, and $\F^\delta_{\tria} \subset \F^\delta_{\tilde{\tria}}$, see Remark \ref{rem:nested}.

\new{Since in the current setting ${\rm osc}_\tria(\bb,c,f)=0$,} according to  \eqref{aposteriori} one has for any  $(u,w) \in \breve{\U}^\delta_\tria$,
\be
 \label{precise2}
\|R^\delta_{\Trs}(u,w;f)\|_{H(\bb;\tria_s)} \eqsim \|(u,w)-(\uex,\uex)\|_{\breve{\U}},
\ee
errors are thus  uniformly equivalent to sums of computable local quantities that suggest themselves as error indicators.

\begin{definition}
\label{def:effective}
 For $r \in \N$ and $\nu\in (0,1)$, we say that a strategy of marking $\cM \subset \tria$ 
is {\em $(r,\nu)$-effective}
when for $\tTr=\tTr(\tria,\cM,r)\in \T$, obtained from $\tria$
 by $r$
refinements of each $K' \in \cM$, and for $n>1$, of each $K' \in \tria$ with  $K' \cap \cup_{K'' \in \cM,\,t\geq0} K''+t \bb \neq \emptyset$,
it holds that 
\be
\label{reduced}
\|R^\delta_{\tTrs}(\tuPG,\twPG;f)\|_{H(\bb;\tilde{\tria}_s)} \le \nu \|R^\delta_{\Trs}(u_\tria^\delta,w_\tria^\delta;f)\|_{H(\bb;\tria_s)},
\ee
where $(\uPG,\wPG), (\tuPG,\twPG)$ are the Petrov-Galerkin solutions of \eqref{PGsol} from $\breve\U^\delta_\Tr, \breve\U^\delta_{\tTr}$,
respectively.
\end{definition}

Note that only for $n>1$ the refinement includes a {\em downstream enrichment} comprised of those cells that are intersected by rays
in direction $\bb$ emanating from cells in $\cM$.
\begin{remark}
\label{rem:coarse}
A repeated application, starting from some initial partition, of mark followed by the downwind enriched refinement strategy,
described in Definition \ref{def:effective},
ensures that no mesh can ever become coarser in the down-stream direction.
\end{remark}

Of course, by \eqref{precise2}, $(r,\nu)$-effectiveness translates \new{for some $\nu' \in (0,1)$} into error decay for the solutions
$$
\|(\uex,\uex)- (\tuPG,\twPG)\|_{ {\U}}\le \nu' \|(\uex,\uex)- (\uPG,\wPG)\|_{ {\U}},
$$
where now $\tTr$ is to be understood as the result of possibly several but uniformly bounded finite number of   refinements of the above type.

As indicated earlier, our goal is to prove effectiveness for a marking strategy based on a bulk-criterion.
To make this precise for some $\vartheta \in (0,1]$, $(u,w)\in \breve{\U}_\tria^\delta$, we let
\be
\label{marked2}
\begin{split}
&\cM=\cM((u,w),\vartheta) \subseteq \Tr \quad \text{be such that} \\
 &  \| R^\delta_{\tria_s({\mathcal M})}(u,w;f) \|_{H(\bb;\tria_s({\mathcal M}))} \ge \vartheta \| R^\delta_{\tria_s}(u,w;f) \|_{H(\bb;\tria_s)}.
 \end{split}
\ee

We are currently able to  fully establish effectivity of the standard bulk chasing strategy based on refining just cells in $\cM$ 
given by \eqref{marked2}, only
in the one-dimensional case. 

\begin{theorem}
\label{thm:red} 
We adopt the assumptions of Theorem \ref{th2new}  with the additional assumption $m_w\leq m_u+1$, and the specifications  \eqref{specified} of $\bb$ and $c$.  
Then, for $n=1$  and $\sigma$ sufficiently small
there exist $r \in\N$, $\nu=\nu(\vartheta) <1$ such that the marking strategy based on \eqref{marked2} 
is $(r,\nu)$-effective for $\tria \in \T$,  $f \in \F^\delta_\tria$.

 Under the forthcoming Conjecture~\ref{conjecture1}, the same result holds true for $n>1$ (thus with the downwind enriched refinement strategy).
\end{theorem}
 
 The remainder of this section is to develop the conceptual ingredients entering results of the above type.
 
 A first natural ingredient for proving Theorem~\ref{thm:red} seems to be Petrov-Galerkin orthogonality \eqref{PGorth}
\begin{align}
\label{ds1}
\|R^\delta_{\tTrs}({u^\delta_{\tilde{\tria}},w^\delta_{\tilde{\tria}};}f)\|_{H(\bb;\tTrs)}^2 &= \|R^\delta_{\tTrs}({u^\delta_{\tria},w^\delta_{\tria};}f)\|_{H(\bb;\tTrs)}^2 \nonumber\\
& \quad - \|t^\delta_{\tTrs}({u_\tria^\delta-u_{\tilde{\tria}}^\delta},{w^\delta_\tria-w^\delta_{\tilde{\tria}}})\|^2_{H(\bb;\tTrs)}.
\end{align}
in combination with a proof of 
$$
\|t^\delta_{\tTrs}({u_\tria^\delta-u_{\tilde{\tria}}^\delta},{w^\delta_\tria-w^\delta_{\tilde{\tria}}})\|_{H(\bb;\tTrs)} \gtrsim \|R^\delta_{\tTrs}({u^\delta_{\tria},w^\delta_{\tria};}f)\|_{H(\bb;\tTrs)}.
$$
A complication, however, is the presence of the ``wrong''  mesh-dependent lifting $R^\delta_{\tTrs}({u^\delta_{\tria},w^\delta_{\tria};}f)$
instead of $R^\delta_{\tria_s}({u^\delta_{\tria},w^\delta_{\tria};}f)$
 in the first term on the right hand side of \eqref{ds1}. To tackle this problem, in the next subsection we construct mesh-independent error indicators. (The appearance of the ``wrong'' norm $\|\cdot \|_{H(\bb;\tTrs)}$ instead of $\| \cdot \|_{H(\bb;\Trs)}$ does not cause any problems
because $\|R^\delta_{\Trs}({u^\delta_{\tria},w^\delta_{\tria};}f)\|_{H(\bb;\tTrs)} = \|R^\delta_{\Trs}({u^\delta_{\tria},w^\delta_{\tria};}f)\|_{H(\bb;\Trs)}$.)

\subsection{A mesh-independent error indicator and related least squares problems}
In the light of the remarks at the end of the previous subsection 
we quantify next the interrelation of various equivalent  error indicators arising from different liftings as well as from different
equivalent inner products. A pivotal role is 
played by the following ``domain-additive'' quantity.
For any subdomain $\Omega' \subseteq \Omega$ we introduce
\be
\label{eta}
\eta^2_{\Omega'} (u,w;f) := \|u-w\|_{L_2(\Omega')}^2 + \|\partial_\bb w +c u- f\|_{L_2(\Omega')}^2.
\ee
Accordingly, for a collection ${\mathcal O}$ of subdomains, we define 
$$
\eta^2_{{\mathcal O}} (u,w;f)=\sum_{\Omega' \in {\mathcal O}} \eta^2_{\Omega'} (u,w;f) = \eta^2_{\bigcup\{\Omega'\in {\mathcal O}\}}(u,w;f).
$$
Note that for $f\in L_2(\Omega)$ and $\Omega'=\Omega$ both components $(u,w)$ of the minimizer of \eqref{eta} over $L_2(\Omega)\times H(\bb;\Omega)$ agree with the
minimizer $w\in H(\bb;\Omega)$ of the {\em least squares} functional
$$
\|\cB w-f\|_{L_2(\Omega)}^2 = \|\partial_\bb w +c w- f\|_{L_2(\Omega')}^2,
$$
see the comment in Section \ref{ssec:5.4} below.

As indicated above it will be crucial to relate these mesh-independent quantities to the following quantities each of which being useful for 
different purposes: 
 {Besides the projected lifted residual from \eqref{s4}, recall first the definitions of the {\em lifted residual}
\begin{align*}
R_{\tria_s}&=R_{\tria_s}(u,w;f)=(R_K)_{K \in \tria_s} \in \V_{\tria_s}=H(\bb;\tria_s),
\intertext{determined by $\langle R_{\tria_s},v\rangle_{H(\bb,\tria_s)}=b_{\tria_s}(u,w|_{\partial\tria_s};v)-\int_\Omega f v\,d{\bf x}$ ($v \in \V_{\tria_s}$).
In a similar spirit as in the analysis of test functions we need to make use of the {\em lifted modified residual}
(cf. \eqref{locallift})}
\brR_{\tria_s}&=\brR_{\tria_s}(u,w;f)=(\brR_K)_{K \in \tria_s} \in \V_{\tria_s},
\intertext{and the piecewise polynomial {\em approximate lifted modified residual} (cf. \eqref{near-opt})}
\brbrR_{\tria_s}&=\brbrR_{\tria_s}(u,w;f)=(\brbrR_K)_{K \in \tria_s} \in \bar{\V}^\delta_{\tria_s}.
\end{align*}
In the current setting of $\bb$ being a constant, and so $d_K\equiv 0$ and $b_{\tria_s}=\breve{b}_{\tria_s}$, the lifted residual and the lifted modified residual differ only in the sense that
$\brR_K$ is the lift of the local residual w.r.t. to the alternative inner product $\sll\cdot,\cdot\srr_{H(\bb;K)}$
on $H(\bb;K)$.

The advantage of the latter quantity is its simple explicit analytic expression from which one can actually see the connection with \eqref{eta}
as the "limit case" with respect to increasing subgrid depth.
In fact, for $K' \in \tria$ we will show that $\|\brbrR_{\tria_s(K')}\|_{H(\bb;\tria_s(K'))}^2 \rightarrow \eta^2_{K'} (u,w;f)$ for the subgrid-depth $1/\sigma$ of the test-search spaces tending to $\infty$.
Since $\brbrR_{\tria_s}$ is constructed as a piecewise polynomial approximation for $\brR_{\tria_s}$, we also have that $\|\brbrR_{\tria_s(K')}\|_{H(\bb;\tria_s(K'))}^2 \rightarrow \|\brR_{\tria_s(K')}\|_{H(\bb;\tria_s(K'))}^2$ when $1/\sigma \rightarrow \infty$. As we will see, the norms $\|\cdot\|_{H(\bb;K)}$ and $\tvert \cdot\tvert_{H(\bb;K)}$ on $H(\bb;K)$ are not only equivalent but even converge to each other when $1/\sigma \rightarrow \infty$, which will yield $\|\brR_{\tria_s(K')}\|_{H(\bb;\tria_s(K'))}^2 \rightarrow \|R_{\tria_s(K')}\|_{H(\bb;\tria_s(K'))}^2$ when $1/\sigma \rightarrow \infty$. Finally, since $R_K^\delta$ is the best approximation to $R_K$ from $\cP_{m_v}(K)$, we have that $\|R_K-R_K^\delta\|_{H(\bb;K)} \leq \|R_K-\brbrR_K\|_{H(\bb;K)} \rightarrow 0$ when $1/\sigma \rightarrow \infty$.} The details of this roadmap are as follows:

\begin{proposition} \label{norms_are_close}For $v \in H(\bb;K)$, we have
$$
\big| \|v\|_{H(\bb;K)}^2- \tvert v\tvert_{H(\bb;K)}^2 \big| \leq |\bb|^{-1} \diam(K) \|v\|_{H(\bb;K)}^2.
$$
\end{proposition}

\begin{proof} W.l.o.g. we consider the case that $\bb/|\bb|=\vec{e}_1$. Given $\vec{x} \in K$, let ${\bf s}$ (${\bf t}$) denote the projection of $\vec{x}$ on $\partial K_{\!-}$ ($\partial K_{\!+}$) along the $x_1$-direction.
Applying $|h z(0)-\int_0^h z(x)\,dx|=|-\int_0^h \int_0^x z'(y)\,dy\,\new{dx}| \leq h\int_0^h |z'(y)|\,dy$, we find that
$$
\big| r({\bf s}) v({\bf s})^2-\int_{s_1}^{t_1} v({\bf x})^2 \,d x_1\big| \leq \frac{2 r({\bf s})}{|\bb|} \int_{s_1}^{t_1} |v({\bf x}) \partial_{\bb} v({\bf x})| d x_1.
$$
Integrating this estimate over $x_2,\ldots,x_n$, using that $d{\bf s}=\frac{|\bb|}{|\bb \cdot {\bf n}_K(s)|}d x_2\ldots d x_n$, and finally applying Cauchy-Schwartz'  inequalities confirms the claim.
\end{proof}

As a consequence, lifted and modified lifted residuals become closer with increasing subgrid depth.
\begin{corollary} \label{corol} For $((u,w),f) \in \breve{\U}_\tria^\delta \times \F_\tria^\delta$ and $K \in \tria_s$, we have
$$
\|R_K-\brR_K\|_{H(\bb;K)} \lesssim |\bb|^{-\frac{1}{2}} \diam(K)^{\frac{1}{2}} \|\brR_K\|_{H(\bb;K)}.
$$
\end{corollary}

\begin{proof}
Inside this proof we drop the subscript $H(\bb;K)$ from the norms and inner products.
Note that for any $v \in H(\bb;K)$, it holds by definition that $\sll \brR_K,v\srr=\langle R_K,v\rangle$.

With $\tau:=\sup_{0 \neq v \in H(\bb;K)} \Big|\frac{\tvert v \tvert^2}{\|v\|^2}-1\Big|$ ($\lesssim |\bb|^{-1} \diam(K)$), we find that
\be \label{s6}
\big|\langle R_K,R_K-\brR_K\rangle\big|=\big|\|R_K\|^2-\tvert \brR_K \tvert^2\big| \leq \tau \|R_K\|^2.
\ee
From
$$
\|R_K\|^2=\sup_{0 \neq v \in H(\bb;K)} \frac{\langle R_K, v\rangle^2}{\|v\|^2}=\sup_{0 \neq v \in H(\bb;K)} \frac{\sll \brR_K, v\srr^2}{\|v\|^2}=
\sup_{0 \neq v \in H(\bb;K)} \frac{\sll \brR_K, v\srr^2}{\tvert v \tvert^2}\frac{\tvert v \tvert^2}{\|v\|^2},
$$
and
$$
\sup_{0 \neq v \in H(\bb;K)} \frac{\sll \brR_K, v\srr^2}{\tvert v \tvert^2}=\tvert \brR_K \tvert^2, \quad\frac{\tvert v \tvert^2}{\|v\|^2}\in [1-{\tau},1+{\tau}],
$$
we infer that
$$
\big| \|R_K\|^2-\tvert \brR_K\tvert^2\big| \leq {\tau} \tvert \brR_K\tvert^2.
$$
Now from $\langle \brR_K,R_K-\brR_K\rangle=\tvert \brR_K\tvert^2- \|\brR_K\|^2$ and \eqref{s6}, we arrive at
$$
\|R_K-\brR_K\|^2 \leq \tau (\tvert \brR_K\tvert^2+\|R_K\|^2)\leq \tau(2+\tau)\tvert \brR_K\tvert^2,
$$
which gives the result.
\end{proof}

Corollary~\ref{corol} is one of the ingredients to prove the mutual closeness of the various error indicators.

\begin{proposition} 
\label{proppie} For $(u,w)\in {{\breve{\U}}^\delta_\tria}$, $f \in \F_\tria^\delta$, $K'\in \tria$, we have
\be \label{s9}
\begin{split}
\sum_{K \in \tria_s(K')} \|R^\delta_{K}- (\partial_{\bbf}w+ \cc u - f)\|_{L_2(K)}^2&+\|\partial_\bb R^\delta_K- (w-u)\|_{L_2(K)} ^2 \\
&\lesssim \sigma^2 \eta^2_{K'}(u,w;f),
\end{split}
\ee
 and
 $$
\big| \| R^\delta_{\tria_s(K')}(u,w;f) \|_{H(\bb;\tria_s(K'))}^2 - \eta^2_{K'}(u,w;f) \big|
\lesssim \sigma\, \eta^2_{K'}(u,w;f),
$$
only dependent on the involved polynomial degrees, and on (upper bounds for) $|\bb|^{-1}$, $\|c\|_{L_\infty(K')}$ and $\varrho$.
\end{proposition}

\begin{proof} Applications of the triangle-inequality show that
\be \label{s1}
\begin{split}
&\sum_{K \in \tria_s(K')} \|R^\delta_{K}- (\partial_{\bbf}w+ \cc u - f)\|_{L_2(K)}^2+\|\partial_\bb R^\delta_K- (w-u)\|_{L_2(K)} ^2 \\
&\leq 2 \|R^\delta_{\tria_s(K')}-\brbrR_{\tria_s(K')}\|_{H(\bb;\tria_s(K'))}^2\\
&\quad+
2 \sum_{K \in \tria_s(K')} \|\brbrR_{K}- (\partial_{\bbf}w+ \cc u - f)\|_{L_2(K)}^2+\|\partial_\bb \brbrR_K- (w-u)\|_{L_2(K)} ^2.
\end{split}
\ee

We estimate now the terms on the righthand side of  \eqref{s1}.
For each $K \in \tria_s(K')$, from $\brbrR_K \in \cP_{m_v}(K)$ and $R_K^\delta$ being the $H(\bb,K)$-orthogonal projection of $R_K$ onto $\cP_{m_v}(K)$,
 we have
\begin{equation*}
\begin{split}
\|R^\delta_K-\brbrR_K\|_{H(\bb;K)} & \leq \|R^\delta_K-R_K\|_{H(\bb;K)}+\|R_K-\brbrR_K\|_{H(\bb;K)} \\
& \leq 2 \|R_K-\brbrR_K\|_{H(\bb;K)} \\
& \leq
2 \|R_K-\brR_K\|_{H(\bb;K)}+2\|\brR_K-\brbrR_K\|_{H(\bb;K)},
\end{split}
\end{equation*}
which yields
\be \label{s7}
\begin{split}
\|R^\delta_{\tria_s(K')}-\brbrR_{\tria_s(K')}\|_{H(\bb;\tria_s(K'))}^2 \leq &
8 \|R_{\tria_s(K')}-\brR_{\tria_s(K')}\|_{H(\bb;\tria_s(K'))}^2\\
&+
8 \|\brR_{\tria_s(K')}-\brbrR_{\tria_s(K')}\|_{H(\bb;\tria_s(K')}^2.
\end{split}
\ee

Using that for $K \in \tria_s(K')$, $\diam(K) \leq \sigma \diam(K') \leq \sigma^2$,
an application of Corollary~\ref{corol} shows that
\be \label{s8}
\|R_{\tria_s(K')}-\brR_{\tria_s(K')}\|_{H(\bb;\tria_s(K'))}^2 \lesssim |\bb|^{-1} \sigma^2 \|\brR_{\tria_s(K')}\|_{H(\bb;\tria_s(K'))}^2.
\ee
Lemma~\ref{lem3} shows that for $K \in \tria_s(K')$,
$$
\|\brR_K-\brbrR_K\|_{H(\bb;K)} \lesssim \diam(K)(\|u-w\|_{H^1(K)}+\|\partial_\bb w +c w- f\|_{H^1(K)})
$$
dependent on (upper bounds for) $\varrho$, $|\bb|^{-1}$, and $\|c\|_{L_\infty(K')}$.
Squaring, summing over $K \subset K'$, and using inverse inequalities yields
\be \label{s3}
\|\brR_{\tria_s(K')}-\brbrR_{\tria_s(K')}\|^2_{H(\bb;{\tria_s(K')})} \lesssim \sigma^2  \eta^2_{K'}(u,w;f).
\ee

It remains to estimate the terms in the sum in the right hand side of \eqref{s1}.
For each $K \in \tria_s(K')$, we have
\begin{align*}
\|\brbrR_{K}- &(\partial_{\bbf}w+ \cc u - f)\|_{L_2(K)}\\
\lesssim &|\bb|^{-1} \diam(K)\big\{
\|w-u\|_{L_2(K)} + |\bb|^{-1}\diam(K)\|\partial_\bb(w-u)\|_{L_2(K)}\big\}\\
&+ |\bb|^{-1} \diam(K) \|\partial_\bb (\partial_{\bb}w+  c u - f)\|_{L_2(K)},
\end{align*}
by applications of Poincar\'{e}'s inequality in the streamline direction (cf. the second paragraph in the proof of Lemma~\ref{lem4}).
Similarly
\begin{align*}
\|\partial_\bb \brbrR_K- (w-u)\|_{L_2(K)} 
 \lesssim |\bb|^{-1} \diam(K)\|\partial_{\bb}(w-u)\|_{L_2(K)}.
\end{align*}
Squaring and summing over $K \in \tria_s(K')$, and using inverse estimates yields
\be \label{s10}
\begin{split}
\sum_{K \in \tria_s(K')} \|\brbrR_{K}- (\partial_{\bbf}w+ \cc u - f)\|_{L_2(K)}^2+\|\partial_\bb \brbrR_K- &(w-u)\|_{L_2(K)} ^2 \\
&\lesssim \sigma^2  \eta^2_{K'}(u,w;f),
\end{split}
\ee
only dependent on (upper bounds for) $\varrho$, $|\bb|^{-1}$ and the involved polynomial degrees.

By combining \eqref{s1}--\eqref{s10} one infers \eqref{s9}.

Now using that for vectors $\vec{a}, \vec{b}$,
$$
\big|\|\vec{a}\|^2-\|\vec{b}\|^2\big| \leq \|\vec{a}-\vec{b}\|\|\vec{a}+\vec{b}\| \leq \|\vec{a}-\vec{b}\|(2\|\vec{b}\|+\|\vec{a}-\vec{b}\|) \leq \|\vec{b}\|^2 {\textstyle \frac{\|\vec{a}-\vec{b}\|}{\|\vec{b}\|}}\big(2+{\textstyle \frac{\|\vec{a}-\vec{b}\|}{\|\vec{b}\|}}\big)
$$
and, when $\vec{a}$ is of the form $(\|f_i\|)_i$ and $\vec{b}=(\|g_i\|)_i$, furthermore
$$
\frac{\|\vec{a}-\vec{b}\|}{\|\vec{b}\|} \leq \frac{\sqrt{\sum_i\|f_i-g_i\|^2}}{\sqrt{\sum_i\|g_i\|^2}},
$$
from \eqref{s9} we conclude that
\begin{align*}
\Big|
\sum_{K \subset K'} \|R^\delta_{K}\|_{L_2(K)}^2+\|\partial_\bb R^\delta_K\|_{L_2(K)} ^2  -
\sum_{K \subset K'} \|\partial_{\bbf}w+ \cc u - f\|_{L_2(K)}^2+\|w-u\|_{L_2(K)} ^2 
\Big|\\
\lesssim \sigma^2 \sum_{K \subset K'} \|\partial_{\bbf}w+ \cc u - f\|_{L_2(K)}^2+\|w-u\|_{L_2(K)} ^2.
\end{align*}
which, in compact notation, is the second statement to be proven.
\end{proof}

\subsection{A companion mesh-independent least squares formulation of the transport problem}
\new{For $(u,w) \in \breve{\U}$, it holds that
\begin{align*}
\eta^2_\Omega(u,w;0)&=\|\partial_\bb w+cu\|^2_{L_2(\Omega)}+\|u-w\|_{L_2(\Omega)}^2 \gtrsim \|\partial_\bb w+cw\|^2_{L_2(\Omega)}+\|u-w\|_{L_2(\Omega)}^2\\
& \gtrsim \|\cB^{-1}\|^{-2}_{\cL(H_{0,\Gamma_{\!-}}(\bb;\Omega),L_2(\Omega))} \|w\|^2_{H(\bb;\Omega)}+\|u-w\|_{L_2(\Omega)}^2\\
& \gtrsim \|(u,w)\|_{\breve{\U}}^2 \gtrsim \eta^2_\Omega(u,w;0).
\end{align*}}
Therefore,  for $f \in L_2(\Omega)$ and any closed subspace of $\breve{\U}$, the problem of minimizing $\eta^2_\Omega(\,,\,;f)$ over that subspace is well-posed.

\begin{proposition} 
\label{solsclose} For $\tria \in \T$, let
\be
\label{LS}
(\bar{u}^\delta_\tria,\bar{w}^\delta_\tria) := \argmin_{(u,w) \in \breve{\U}_\tria^\delta} \eta_\Omega^2(u,w;f).
\ee
Then for $\sigma$ small enough, it holds that 
$$
\|(u_\tria^\delta,w_\tria^\delta)-(\bar{u}_\tria^\delta,\bar{w}_\tria^\delta)\|_{\breve{\U}}^2 \lesssim \sigma  \|(\uex,\uex)-(u_\tria^\delta,w_\tria^\delta)\|^2_{\breve{\U}},
$$
where $(\uPG,\wPG)\in \breve\U^\delta_\Tr$ is the Petrov-Galerkin solution of \eqref{PGsol}
\end{proposition}
\begin{proof} 
`Galerkin orthogonality' shows that for any $(u,w) \in \breve{\U}_\tria^\delta$,
\be \label{201}
\eta^2_\Omega(u,w;f)-\eta^2_\Omega(\bar{u}^\delta_\tria,\bar{w}^\delta_\tria;f)=\eta^2(u-\bar{u}_\tria^\delta,w-\bar{w}_\tria^\delta;0) \eqsim
\|(u,w)-(\bar{u}_\tria^\delta,\bar{w}_\tria^\delta)\|_{\breve{\U}}^2.
\ee
Since $(u_\tria^\delta,w_\tria^\delta)$
minimizes $\|R_{\tria_s}^\delta(u,w;f)\|^2_{H(\bb;\tria_s)}$ over $(u,w) \in \breve{\U}_\tria^\delta$, two applications of Proposition~\ref{proppie} show that for some $|\xi_1|, |\xi_2| \lesssim \sigma$
\begin{align*}
(1+\xi_1) \eta_\Omega^2(u_\tria^\delta,w_\tria^\delta;f)&=\|R_{\tria_s}^\delta(u_\tria^\delta,w_\tria^\delta;f)\|^2_{H(\bb;\tria_s)} \\
&\leq 
\|R_{\tria_s}^\delta(\bar{u}_\tria^\delta,\bar{w}_\tria^\delta;f)\|^2_{H(\bb;\tria_s)}  = (1+\xi_2) \eta_\Omega^2(\bar{u}_\tria^\delta,\bar{w}_\tria^\delta;f),
\end{align*}
which, together with \eqref{201}, shows that for $\sigma$ small enough,
$$
\|(u_\tria^\delta,w_\tria^\delta)-(\bar{u}_\tria^\delta,\bar{w}_\tria^\delta)\|_{\breve{\U}}^2 \lesssim \sigma \eta_\Omega^2(u_\tria^\delta,w_\tria^\delta;f) \eqsim \sigma  \|(\uex,\uex)-(u_\tria^\delta,w_\tria^\delta)\|^2_{\breve{\U}}.\qedhere
$$
\end{proof}

 In complete analogy we can define effectivity of a mark-and-refine strategy for the least squares scheme \eqref{LS} based on a bulk criterion for the
quantities $\eta_K$, denoting the collection of correspondingly marked cells by $\bar\cM= \bar{\cM}((\bar{u}_\tria^\delta,\bar{w}_\tria^\delta),\vartheta)$.
\begin{proposition}
\label{rem:equiv}
For sufficiently small $\sigma$,  {$(\refs,\nu)$}-effectivity of the above refinement strategy for the DPG-scheme is {\em equivalent} to  {$(\refs,\nu)$}-effectivity of the analogous strategy  {with the same $\vartheta$} for the least squares estimator.
\end{proposition}

 \begin{proof}
Using Proposition~\ref{solsclose}, stability of both estimators shows that for any $\cM \subset \tria$,
 \be \label{203}
 \begin{split}
& \left. \begin{array}{r}
 \big| \|R_{\tria_s(\cM)}^\delta(\bar{u}_\tria^\delta,\bar{w}_\tria^\delta;f)\|_{H(\bb;\tria_s(\cM))} -
 R_{\tria_s(\cM)}^\delta(u_\tria^\delta,w_\tria^\delta;f)\|_{H(\bb;\tria_s(\cM))} \big|\\[2mm]
 \big| \eta_{\cM}(\bar{u}_\tria^\delta,\bar{w}_\tria^\delta;f)-
 \eta_{\cM}(u_\tria^\delta,w_\tria^\delta;f)\big|
\end{array}
\right\}\\
 &\qquad\lesssim \sqrt{\sigma}  \eta_\Omega(u_\tria^\delta,w_\tria^\delta;f).
 \end{split}
 \ee
  Now 
  let $\cM \subset \tria$ be such that
 $$
 \|R_{\tria_s(\cM)}^\delta(u_\tria^\delta,w_\tria^\delta;f)\|_{H(\bb;\tria_s(\cM))} \geq \vartheta \|R_{\tria_s}^\delta(u_\tria^\delta,w_\tria^\delta;f)\|_{H(\bb;\tria_s)}.
 $$
 Then elementary operations using  Propositions~\ref{proppie} and \ref{solsclose} show the existence of a $|\xi|\lesssim \sqrt{\sigma}$, and thus for $\sigma$ small enough,  $|\xi|\leq \frac12$, with
 $$
  \eta_{\cM}(\bar{u}_\tria^\delta,\bar{w}_\tria^\delta;f) \geq \vartheta (1+\xi) \eta_{\Omega}(\bar{u}_\tria^\delta,\bar{w}_\tria^\delta;f).
  $$
Now, \emph{if} the latter implies that for some $\nu=\nu(\vartheta)<1$, and with  {the refined mesh $\tilde{\tria}=\tilde{\tria}(\tria,\cM,r)$ from Definition~\ref{def:effective}},
it holds that $\eta_\Omega(\bar{u}_{\tilde \tria}^\delta,\bar{w}_{\tilde \tria}^\delta;f) \leq \nu \eta_\Omega(\bar{u}_\tria^\delta,\bar{w}_\tria^\delta;f)$, \emph{then} we have that for some $|\xi_1|,|\xi_2|,|\xi_3| \leq \sigma$,
\begin{align*}
&\| R^\delta_{\tilde{\tria}_s}(u_{\tilde \tria}^\delta,w_{\tilde \tria}^\delta;f) \|_{H(\bb;\tilde{\tria}_s)} \leq
\| R^\delta_{\tilde{\tria}_s}(\bar{u}_{\tilde \tria}^\delta,\bar{w}_{\tilde \tria}^\delta;f) \|_{H(\bb;\tilde{\tria}_s)}
= \eta_\Omega(\bar{u}_{\tilde \tria}^\delta,\bar{w}_{\tilde \tria}^\delta;f)(1+\xi_1)
\\
&\leq \nu \eta_\Omega(\bar{u}_\tria^\delta,\bar{w}_\tria^\delta;f)(1+\xi_1)=\nu \eta_\Omega(u_\tria^\delta,w_\tria^\delta;f)(1+\xi_1)(1+\sqrt{\xi_2})\\
&=
\nu \|R_{\tria_s}^\delta(u_\tria^\delta,w_\tria^\delta;f)\|_{H(\bb;\tria_s)}(1+\xi_1)(1+\sqrt{\xi_2})(1+\xi_3),
\end{align*}
showing for $\sigma$ small enough the result of Theorem~\ref{thm:red}.

Applying the above arguments with interchanged roles of $\|R_{\tria_s}^\delta(\,,\,;f)\|_{H(\bb;\tria_s)}$ and $\eta_\Omega(\,,\;f)$ and  
choosing $\sigma$ small enough, the claim of Remark \ref{rem:equiv} follows.
\end{proof}

In view of  {Proposition} \ref{rem:equiv}, the proof of Theorem \ref{thm:red}
  {is} complete once we establish the following equivalent
result.

 {\begin{theorem} 
\label{thm:red2} 
Let $\bb$ and $c$ be as in Theorem~\ref{thm:red}, and, to control $\,\sup_{K \in \tria \in \T}\diam(K)$, let $\sigma$ be sufficiently small. Then
for all $\vartheta \in (0,1]$, there exist $r \in\N$, $\nu=\nu(\vartheta) <1$ with the following property: whenever
for $\tria \in \T$ and $f \in \F^\delta_\tria$, the set of marked elements
$\bar{\cM}=\bar{\cM}((\bar{u}_\tria^\delta,\bar{w}_\tria^\delta),\vartheta) \subseteq \tria$ is such that
\be \label{bulk}
\eta_{\bar{\cM}}(\bar{u}_\tria^\delta,\bar{w}_\tria^\delta;f) \geq \vartheta \eta_\Omega(\bar{u}_\tria^\delta,\bar{w}_\tria^\delta;f),
\ee
then for the refinement $\tTr=\tTr(\tria,\bar{\cM},r)$ according to Definition~\ref{def:effective}, it follows that
\be
\label{reduced2}
 \eta_\Omega(\bar{u}_{\tilde \tria}^\delta,\bar{w}_{\tilde \tria}^\delta;f) \leq \nu \eta_\Omega(\bar{u}_\tria^\delta,\bar{w}_\tria^\delta;f).
\ee
\end{theorem}

The remainder of this section is devoted to the proof of 
 Theorem~\ref{thm:red2}.
  We are going to show that for some constants $\vartheta' >0$ and $\nu' <1$, thus independent of $\tria$ (subject to $\sigma$ being
  sufficiently small), for $\bar{\cM}$ as in \eqref{bulk} there exists an $\bar{\bar{\cM}} \subset \bar{\cM}$ with 
 \be \label{fraction}
\eta_{\bar{\bar{\cM}}}(\bar{u}_\tria^\delta,\bar{w}_\tria^\delta;f) \geq \vartheta'\eta_{\bar{\cM}}(\bar{u}_\tria^\delta,\bar{w}_\tria^\delta;f),
\ee
and that for any $K'\in \bar{\bar\cM}$,
 \be \label{eta_reduction}
\inf_{\{(u,w) \in \breve{\U}^\delta_{\tilde \tria}\colon \supp u,\,\supp w \subset K'\}} \eta_{K'}(\bar{u}_\tria^\delta-u,\bar{w}_\tria^\delta-w;f) \leq \nu' \eta_{K'}(\bar{u}_\tria^\delta,\bar{w}_\tria^\delta;f).
\ee
In other words, for the cells in $\bar{\bar{\cM}}$ one can correct the current approximation {\em cell-wise} to reduce the corresponding
error indicator.
An elementary calculation shows that then these two properties imply \eqref{reduced2} with constant\\ $\nu:=\sqrt{(\vartheta\vartheta')^2 (\nu')^2+1-(\vartheta\vartheta')^2}<1$.}

\subsubsection{Reduction of the local mesh-independent error indicator}\label{ssec:5.4}
In this subsection we work towards the verification of  \eqref{eta_reduction} for those $K' \in  \bar{\cM}$ that satisfy certain conditions. 
Then in the following two subsections, for two possible scenarios we will construct subsets $\bar{\bar{\cM}} \subset \bar{\cM}$ of $K'$ that satisfy these conditions, and for which \eqref{fraction} is satisfied.  This will then prove Theorem~\ref{thm:red2} and hence Theorem~\ref{thm:red}.

We recall that the  {reaction coefficient $c$ is assumed to be a non-negative constant} over each $K'\in \Tr$.
We introduce the shorthand notations
\be \label{shortnots}
g:=  {\partial_\bb \bar{w}^\delta_\tria +c \bar{u}^\delta_\tria-f},\quad  {e:=\bar{u}^\delta_\tria-\bar{w}^\delta_\tria},
\ee
so that
$$
\eta_{K'}^2(\bar{u}_\tria^\delta-u,\bar{w}^\delta_\tria-w;f)=\|e-(u-w)\|_{L_2(K')}^2+\|g-(\partial_\bb w+c u)\|_{L_2(K')}^2.
$$

Fixing
$$
 {\beta \in \big(0,\frac{1}{4}\big),}
$$
we refer to the $K' \in  {\bar{\cM}}$ for which
\be
\label{type-I}
\frac{\|e+cg\|_{L_2(K')}^2}{\|g\|_{L_2(K')}^2+\|e\|_{L_2(K')}^2}\geq \beta \quad\text{(Type-(I))},
\ee
as Type-(I) and for the remaining ones as Type-(II). Accordingly, we decompose $\bar\cM$ into the Type-(I) and Type-(II) elements
 writing $\bar{\mathcal M}=\bar{\mathcal M}_{\rm I} \dot\cup \bar{\mathcal M}_{\rm II}$.

\paragraph{\bf Type-(I) elements:}
We start with showing that for  $K' \in \bar{\mathcal M}_{\rm I}$, \eqref{eta_reduction} can be already established by a correction of the $u$-component.
\begin{lemma} 
\label{lem:aux}
Assume that $\|\cdot\|$ is induced by the inner product $\langle\cdot,\cdot\rangle$ of some Hilbert space $H$ and let $g,e\in H$ be arbitrary but fixed.
For any scalar $c$ and $u\in H$ let
\be
\label{Qc}
Q(u):= \|e-u\|^2+\|g-c u\|^2.
\ee
Then
\begin{align*}
u_{\min}:=\argmin_{u \in H} Q(u)&=\frac{e+c g}{1+c^2},\\
Q(u)-Q(u_{\min})&=(1+c^2)\|u-u_{\min}\|_H^2,\\
\|u_{\min}\|^2 &\leq \frac{Q(0)}{1+c^2},\\
Q(u_{\min}) &= \Big(1-\frac{\|e+cg\|^2}{(1+c^2)(\|g\|^2+\|e\|^2)}\Big) Q(0),
\end{align*}
\end{lemma}

\begin{proof}
The first two statements follow from
$$
Q(u+h)-Q(u)=2\langle h,(c^2+1)u-(e+cg)\rangle+(1+c^2)\|h\|^2.
$$
The third statement is a consequence of
$$
\|u_{\min}\|=\|\frac{e+cg}{1+c^2}\|\leq \frac{1}{1+c^2}\|e\|+\frac{c}{1+c^2}\|g\| \leq \frac{\sqrt{\|e\|^2+\|g\|^2}}{\sqrt{1+c^2}}=\frac{Q(0)^{\frac{1}{2}}}{\sqrt{1+c^2}}.
$$
The last statement follows from
$$
\frac{Q(0)-Q(u_{\min})}{Q(0)} =\frac{(1+c^2)\|\frac{e+cg}{1+c^2}\|^2}{\|g\|^2+\|e\|^2}.\qedhere
$$
\end{proof}

\begin{corollary}
\label{lem:muc}
For $\refs$ sufficiently large, only dependent on the polynomial degrees $m_u$, $m_w$ and $m_f$, and on  {an upper bound for $|c_{K'}|$},
for all  $K' \in \bar{\mathcal M}_{\rm I}$ it holds that
$$
\inf_{\{(u,0) \in \breve{\U}_{\tilde \tria}^\delta\colon \supp u\subset K'\}} \eta_{K'}^2(\bar{u}_\tria^\delta-u,\bar{w}^\delta_\tria;f) \leq \Big(1-\frac{\beta}{2(1+c_{K'}^2)}\Big) \eta_{K'}^2(\bar{u}_\tria^\delta,\bar{w}^\delta_\tria;f).
$$
\end{corollary}

\begin{proof}
Lemma \ref{lem:aux} says that $u_{\min}=\frac{e+c_{K'} g}{1+c_{K'}^2}$ minimizes $Q(u):=\eta^2_{K'}(\bar{u}^\delta_\tria-u,\bar{w}^\delta_\tria;f)$ over $L_2(K')$, and that 
$Q(u_{\min}) \leq (1-\frac{\beta}{1+c_{K'}^2}) Q(0)$.

The function $u_{\min}$ is a polynomial on $K'$ and can therefore be  approximated with relative accuracy $\sqrt{\frac{\beta}{2(1+c_{K'}^2)}}$ by a 
piecewise polynomial $\tilde u$ on a sufficiently refined mesh. This follows from the usual combination of direct and inverse estimates.
The proof is completed by
$$
Q(\tilde u)-Q(u_{\min})=(1+c_{K'}^2)\|\tilde u-u_{\min}\|_{L_2(K')}^2 \leq \beta/2 \|u_{\min}\|_{L_2(K')}^2 \leq \frac{\beta}{2(1+c_{K'}^2)} Q(0)
$$
by applications of the statements from Lemma~\ref{lem:aux}.
\end{proof} 

\paragraph{\bf Type-(II) elements:}
It remains to discuss $K' \in \bar{\mathcal M}_{\rm II}$.
For those elements we need to find suitable corrections for the component $\bar{w}^\delta_\tria$ - in brief the $w$-component.

We will search for a $(0,w) \in \breve{\U}_{\tilde{\tria}}^\delta$ with $\supp w \subset K'$ such that $\|g-\partial_\bb w\|_{L_2(K')}^2<\|g\|_{L_2(K')}^2$.
In order to show that this reduction is not lost by a similar increase by the replacement of $\|e\|_{L_2(K')}^2$ by $\|e-w\|_{L_2(K')}^2$, we will make use of the fact that for $K' \in \bar{\mathcal M}_{\rm II}$, the term $\|e\|_{L_2(K')}$ is controlled
by a multiple of $\|g\|_{L_2(K')}$ depending only on $\|c\|_{L_\infty(\Omega)}$:
\begin{lemma} \label{lem1}
 For $K'\in \bar{\mathcal M}_{\rm II}$, it holds that 
\be
\label{mug}
 {\omega_{K'} =  \omega_{K'}(e,g)} := \frac{\|e\|_{L_2(K')}}{\|g\|_{L_2(K')}}  {<} 2 |c_{K'}| + 1,
\ee
and thus  {$\eta_{K'}^2(\bar{u}_\tria^\delta,\bar{w}_\tria^\delta;f)^2 \leq \big((2 |c_{K'}| + 1)^2+1\big) \|g\|^2_{L_2(K')}$.}
\end{lemma}

\begin{proof}
 Recall that $K'\in \bar{\mathcal M}_{\rm II}$ means that 
\be
\label{type-II}
\frac{\|e+cg\|_{L_2(K')}^2}{\|g\|_{L_2(K')}^2+\|e\|_{L_2(K')}^2}< \beta,
\ee
so that in particular $g \neq 0$. Substituting $\|e\|_{L_2(K')} =  { {\omega_{K'}}} \|g\|_{L_2(K')}$,
\eqref{type-II} implies
$\big| { {\omega_{K'}}} - |c_{K'}|\big|  {<} \sqrt{\beta(1+ { {\omega_{K'}}}^2)}$
which gives
\begin{align*}
 { {\omega_{K'}}} & {<} |c_{K'}| + \sqrt{\beta} + \sqrt{\beta}\, { {\omega_{K'}}}  {<} |c_{K'}| + {\textstyle \frac{1}{2}}+ {\textstyle \frac{1}{2}}\, { {\omega_{K'}}}
\end{align*}
 {by our assumption that $\beta<\frac{1}{4}$}.  {This confirms the first and so the second claim.}
\end{proof}

Our argument for handling Type-(II) elements requires the following further preparations.
  For every  $\bs\in \partial K'_-$  
let as before $r(\bs)$ denote length of the line segment emanating from $\bs\in \partial K'_-$ and ending in $\partial K'_+$.
One observes then
that a function $Q$ on $K'$ can be written as 
\be
\label{Qz}
Q = \partial_\bb z, \quad  z|_{\partial K'_-\cup\partial K'_+}=0,
\ee
 if and only if each of its
{\em line averages  in direction } $\bbb := \bb/|\bb|$ vanishes, i.e.,
$$
A_\bs(Q):= r(\bs)^{-1}\int_0^{r(\bs)}Q(\bs+t\bbb)dt =0,\quad \bs\in \partial K_-.
$$
In fact, then
$
z(\bs + t\bbb) :=|\bb|^{-1} \int_0^t Q(\bs + t'\bbb)dt'
$
satisfies \eqref{Qz}. 

For $g$ as in \eqref{shortnots}, the function  {$G=G(g)$, defined on each $K' \in \tria$ by}
\be \label{defG}
\GG({\bf x}) = A_\bs(g)\quad \mbox{for } \,\, {\bf x} = \bs + t\bbb,\, \bs\in \partial K_-,\,  {t\in [0,r(s)],}
\ee
is obviously constant along $\bb$ and
\be
\label{Q}
A_\bs (g - \GG) = 0, \quad\mbox{every $\bs\in \partial K'_-$}.
\ee
Hence, for $z_g$, defined by
\be
\label{zdef}
z_g(\bs + t\bbb) :=|\bb|^{-1} \int_0^t  (g - \GG)(\bs + t'\bbb)dt'\quad \mbox{for } \,\,  {t \in [0,r(s)],}
\ee
we have 
\be
\label{dbz}
g-\partial_\bb z_g = \GG.
\ee

Thanks to $g-\GG \perp_{L_2(K')} \GG$, we have 
$$
\|\GG\|^2_{L_2(K')}=\|g\|_{L_2(K')}^2-\|g-\GG\|_{L_2(K')}^2,
$$
and so in particular $\|\GG\|_{L_2(K')} \leq \|g\|_{L_2(K')}$.

Under the condition that  $\|\GG\|_{L_2(K')}<\|g\|_{L_2(K')}$, one infers 
from 
$$
\|z_g\|_{L_2(K')} \lesssim  {|\bb|^{-1}} {\rm diam}\,K' \|g- \GG\|_{L_2(K')}
$$
by   {Poincar\'{e}'s inequality, in combination with \eqref{mug} that for $\diam K'$ being sufficiently small,
$\eta_\Omega^2(\bar{u}_\tria^\delta,\bar{w}^\delta_\tria-z_g;f)<\eta_\Omega^2(\bar{u}_\tria^\delta,\bar{w}^\delta_\tria;f)$.

When proceeding to the natural next step to approximate $z_g$ with functions of type $(0,w) \in \breve{\U}_{\tilde \tria}^\delta$ with $\supp w \subset K'$, a difficulty is that $z_g$ is continuous piecewise polynomial w.r.t. a partition of $K'$ into subsimplices that can have arbitrarily bad aspects ratios.
To tackle this problem, we first approximate $z_g$ by an `isotropic' function $\tilde{z}_g$ for which $\|g-\partial_\bb \tilde{z}_g\|_{L_2(K')}$ is at most slightly larger than $\|g-\partial_\bb z_g\|_{L_2(K')}$:

\begin{lemma} \label{lem10} Let 
\be
\label{ratio}
 {\alpha_{K'} = \alpha_{K'}(g)} := \frac{\|\GG\|_{L_2(K')}}{\|g\|_{L_2(K')}}<1.
\ee
Then there exists a $\tilde{z}_g \in H^1_0(K') \cap H^s(K')$ such that for any $s<\frac{3}{2}$,
\be 
\label{inverse_ineq}
|\tilde{z}_g|_{H^s(K')} \lesssim (\diam K')^{-s} \|\tilde{z}_g\|_{L_2(K')}
\ee
{\rm(}depending on upperbounds for ${\alpha_{K'}}$ and $\varrho_{K'}${\rm)}, and
$$
\|g-\partial_\bb \tilde{z}_g\|_{L_2(K')} \leq \frac{1+ {\alpha_{K'}}}{2}\|g\|_{L_2(K')}.
$$
\end{lemma}

\begin{proof} For $n=1$, $\tilde{z}_g=z_g$ satisfies the conditions. Now let $n>1$.
Let $\rho \in C^\infty$ with $0 \leq \rho \leq 1$, $\rho(x)=0$ for $x \leq \frac12$, and $\rho(x)=1$ for $x \geq 1$, and let $\rho_\eta(x):=\rho(x/\eta)$.

We are going to construct a modification of $z_g$ that is zero on subsimplices that have very bad aspect ratios.
With $F_1,\ldots,F_{n+1}$ denoting the faces of $K'$, for $1 \leq i \leq n+1$ let ${\bf d}_{F_i}$ be the orthogonal projection of the inward pointing normal to $F_i$ onto the plane $\bb^\perp$.
 {For each $i$, we choose a Cartesian coordinate system ${\bf y}^{(i)}=T^{(i)} {\bf x} +{\bf z}^{(i)}$ such that the first coordinate direction is ${\bf d}_{F_i}/|{\bf d}_{F_i}|$, the origin equals one of the vertices of $F_i$, and all other vertices of $F_i$ have a non-negative first component.}
Now for some $\eps>0$, we define $\tilde{z}_g$ by
$$
\partial_\bb \tilde{z}_g=(g-\GG)\prod_{i=1}^{n+1} \rho_{\eps \diam K'}((T^{(i)}\cdot  {+{\bf z}^{(i)}})_1), \quad \tilde{z}_g|_{\partial K'_{-} \cup \partial K'_{+}}=0.
$$
Since $(T^{(i)} (\cdot+t \bb))_1=(T^{(i)} (\cdot))_1$,  $A_{\bs}(\partial_\bb\tilde z_g)=0$ and the function $\tilde{z}_g$ is well-defined.

Since ${\bf x} \mapsto \prod_{i=1}^{n+1} \rho_{\eps \diam K'}((T^{(i)}{\bf x})_1)$ vanishes on all subsimplices that have very bad aspect ratios 
(relative to $\eps$)
in the partition of $K'$ w.r.t. which $z_g$ is a continuous piecewise polynomial, a homogeneity argument shows that $\tilde{z}_g$ satisfies \eqref{inverse_ineq}, with a constant depending on $\eps$. Moreover, also $\tilde{z}_g$ vanishes on a possible characteristic boundary of $K'$.

Writing $g-\partial_\bb \tilde{z}_g=\GG+\Big(1-\prod_{i=1}^{n+1} \rho_{\eps \diam K'}((T^{(i)}\cdot)_1)\Big) (g-\GG)$, and using that $\|\GG\|_{L_2(K')}= {\alpha_{K'}} \|g\|_{L_2(K')}$,
and
\begin{align*}
&\|\Big(1-\prod_{i=1}^{n+1} \rho_{\eps \diam K'}((T^{(i)}\cdot)_1)\Big) (g-\GG)\|_{L_2(K')}   \\
&\leq \|\Big(1-\prod_{i=1}^{n+1} \rho_{\eps \diam K'}((T^{(i)}\cdot)_1)\Big)\|_{L_2(K')} \| g-\GG\|_{L_\infty(K')} \\
& \lesssim \sqrt{\eps |K'|}\, 2 \| g\|_{L_\infty(K')} \lesssim \sqrt{\eps} \, \|g\|_{L_2(K')},
\end{align*}
which holds again by a homogeneity argument, the proof is completed by taking $\eps$ sufficiently small, dependent on ${\alpha_{K'}}$.
\end{proof}

\begin{corollary} 
\label{corol2}  For $K' \in \bar{\mathcal M}_{\rm II}$ let $ {\alpha_{K'}}<1$. Then for $\sigma$ sufficiently small, and $\refs$ sufficiently large, only dependent on upperbounds for $m_u$, $m_w$, $m_f$, $\varrho$, $|\bb|^{-1}$, $ {\alpha_{K'}}$, $\sigma$, and $|c_{K'}|$,  it holds that
$$
\inf_{\{w\colon \supp w \subset K',\, (0,w) \in \breve{\U}_{\tilde \tria}^\delta\}\hspace*{-1em}} \eta_{K'}^2(\bar{u}_\tria^\delta,\bar{w}^\delta_\tria-w;f) \leq 
 {\Big({\textstyle \frac12+\frac12}\frac{\frac{1+ {\alpha_{K'}}}{2}+|c_{K'}|+1}{1+|c_{K'}|+1}\Big)}
 \eta_{K'}^2(\bar{u}_\tria^\delta,\bar{w}^\delta_\tria;f).
$$
\end{corollary}

\begin{proof} Let $\breve{\sigma}=\breve{\sigma}(\refs):=\max_{\{K \in \tilde \tria\colon K \subset K'\}} \frac{\diam K}{\diam K'}$. By taking $w$ with $(0,w)\in \breve{\U}_{\tilde \tria}^\delta$ to be the Scott-Zhang interpolant of $\tilde{z}_g$ from Lemma~\ref{lem10}, for $s\in (1,\frac{3}{2})$ we have
\begin{align*}
\|\tilde{z}_g-w\|_{L_2(K')} + |\bb|^{-1} \breve{\sigma}\diam K'  \|\partial_\bb(\tilde{z}_g-w)\|_{L_2(K')} \lesssim (\breve{\sigma}\diam K')^s |\tilde{z}_g|_{H^s(K')}\\
\lesssim \breve{\sigma}^s \| \tilde{z}_g\|_{L_2(K')} \lesssim \breve{\sigma}^s |\bb|^{-1} \diam K' \|\partial_\bb  \tilde{z}_g\|_{L_2(K')} \lesssim \breve{\sigma}^s |\bb|^{-1} \diam K' \|g\|_{L_2(K')},
\end{align*}
where we used Poincar\'{e}'s inequality.
We obtain that
\begin{align*}
\|g-\partial_\bb w\|_{L_2(K')} &\leq  \|g-\partial_\bb \tilde{z}_g\|_{L_2(K')}+\|\partial_\bb (\tilde{z}_g-w)\|_{L_2(K')}\\
& \leq \big({\textstyle \frac{1+ {\alpha_{K'}}}{2}}+\breve{\sigma}^{s-1}\big) \|g\|_{L_2(K')},
\end{align*}
and
\begin{align*}
\|e+w\|_{L_2(K')}-\|e\|_{L_2(K')}&\leq \|\tilde{z}_g\|_{L_2(K')}+\|\tilde{z}_g-w\|_{L_2(K')} \\
&\lesssim (|\bb|^{-1} \diam K'+\breve{\sigma}^s |\bb|^{-1} \diam K') \|g\|_{L_2(K')}.
\end{align*}
Recalling that $\max_{K' \in \tria}\diam K' \leq \sigma$, $\eta_{K'}^2(\bar{u}_\tria^\delta,\bar{w}^\delta_\tria;f)=\|g\|^2_{L_2(K')}+\|e\|_{L_2(K')}^2$, and $ {\omega_{K'}}=\frac{\|e\|_{L_2(K')}}{\|g\|_{L_2(K')}} \leq 2|c_{K'}|+1$, the assertion follows.
\end{proof}}

 In summary,  for $K' \in \bar{\mathcal M}_{\rm I}$ completely local $u$-corrections on refinements of fixed depth suffice to reduce $\eta_{K'}$
 by a constant factor $\nu'<1$.
For  $K' \in \bar{\mathcal M}_{\rm II}$ an analogous statement, this time by  means of a {\em local $w$-correction}, holds \emph{provided} that 
there exists a  constant $\alpha < 1$ such that  
\be
\label{alphawish}
 {\alpha_{K'}}=\frac{\|\GG\|_{L_2(K')}}{\|g\|_{L_2(K')}}=\frac{\sqrt{\langle \GG,g\rangle_{L_2(K')}}}{\|g\|_{L_2(K')}}\le \alpha.
\ee

\subsection{Selection of $\bar{\bar{\cM}} \subset \bar{\cM}$ that satisfy both \eqref{eta_reduction} and \eqref{fraction}} \label{ssec:selection}
In case
\be \label{firstcase}
\eta_{\bar {\mathcal M}_{\rm II}}(\bar{u}_\tria^\delta,\bar{w}_\tria^\delta;f)^2 < \eta_{\bar {\mathcal M}_{\rm I}}(\bar{u}_\tria^\delta,\bar{w}_\tria^\delta;f)^2, 
\ee
equation \eqref{fraction} is valid with $\bar{\bar{\mathcal M}}=\bar {\mathcal M}_{\rm I}$ and $\vartheta'= {\frac12\sqrt{2}}$, whereas   \eqref{eta_reduction} follows from the reduction of the $\eta_{K'}$ for $K' \in 
\bar {\mathcal M}_{\rm I}$ by Corollary~\ref{lem:muc}. We conclude that Theorem~\ref{thm:red2} is valid
    for both $n=1$ and $n>1$  {(even without the additional downwind refinements described in Definition~\ref{def:effective}).}

It remains to investigate the case where \eqref{firstcase} does not hold. It is only for this case that we have to establish \eqref{alphawish} for sufficiently many $K' \in \bar{\mathcal M}_{\rm II}$. It will require `global' arguments, already announced in the abstract, that make use of the fact that $(\bar{u}_\tria^\delta,\bar{w}_\tria^\delta)$ is the minimizer of $\eta_\Omega^2(u,w;f)$ over $\breve{\U}^\delta_\tria$.

\begin{lemma} \label{lem13}
 Suppose there exists a constant $\alpha<1$ such that validity of
\be \label{secondcase}
\eta_{\bar {\mathcal M}_{\rm II}}(\bar{u}_\tria^\delta,\bar{w}_\tria^\delta;f)^2  {\geq}  \eta_{\bar {\mathcal M}_{\rm I}}(\bar{u}_\tria^\delta,\bar{w}_\tria^\delta;f)^2,
\ee
implies
\be \label{n4}
\sum_{K' \in \bar {\mathcal M}_{\rm II}} \|\GG\|^2_{L_2(K')} \leq \alpha^2 \sum_{K' \in \bar {\mathcal M}_{\rm II}} \|g\|^2_{L_2(K')}.
\ee
 {Then Theorem \ref{thm:red2} is valid.}
\end{lemma}

\begin{proof} 
 In view of the discussion preceding this lemma, it suffices to verify \eqref{fraction} and \eqref{eta_reduction} for some $\bar{\bar{\mathcal M}} \subset \bar{\mathcal M}$ for the case that \eqref{secondcase} holds. By the hypothesis 
of this lemma \eqref{n4} is then also valid.
We define 
\be
\label{Mdbar}
\bar{\bar{\mathcal M}}:=\Big\{K' \in \bar {\mathcal M}_{\rm II}\colon  {\alpha_{K'}}\leq {\textstyle \sqrt{\frac{1+\alpha^2}{2}}}\Big\}
\ee
Then $\bar{\bar{\mathcal M}}$ satisfies \eqref{eta_reduction} by Corollary~\ref{corol2}, and it remains to verify that it satisfies \eqref{fraction}.

Thanks to \eqref{secondcase}, we have $\eta_{\bar {\mathcal M}}(\bar{u}_\tria^\delta,\bar{w}_\tria^\delta;f)^2 \leq 2 \eta_{\bar {\mathcal M}_{\rm II}}(\bar{u}_\tria^\delta,\bar{w}_\tria^\delta;f)^2$, whereas by Lemma~\ref{lem1}, the right-hand side is bounded by a constant multiple of $\sum_{K' \in \bar {\mathcal M}_{\rm II}} \|g\|_{L_2(K')}^2$.
The definition of $\bar{\bar{\mathcal M}}$ and \eqref{n4} imply that
$$
\sum_{K' \in \bar {\mathcal M}_{\rm II} \setminus \bar{\bar{\mathcal M}}} \|g\|_{L_2(K')}^2
<{\textstyle \frac{2}{1+\alpha^2}}
\sum_{K' \in \bar {\mathcal M}_{\rm II} \setminus \bar{\bar{\mathcal M}}} \|G\|_{L_2(K')}^2
\leq {\textstyle  \frac{2 \alpha^2}{1+\alpha^2}}
\sum_{K' \in \bar {\mathcal M}_{\rm II}} \|g\|_{L_2(K')}^2,
$$
or, equivalently,
$$
\sum_{K' \in \bar {\mathcal M}_{\rm II}} \|g\|_{L_2(K')}^2 <{\textstyle \frac{1+\alpha^2}{1-\alpha^2}}\sum_{K' \in \bar{\bar{\mathcal M}}} \|g\|_{L_2(K')}^2.
$$
The proof of \eqref{fraction} follows from $\sum_{K' \in \bar{\bar{\mathcal M}}} \|g\|_{L_2(K')}^2 \leq \eta_{\bar{\bar{\mathcal M}}}(\bar{u}_\tria^\delta,\bar{w}_\tria^\delta;f)^2$.
\end{proof}

\subsection{Proof of Theorem \ref{thm:red2} for $n=1$}\label{ssec:proof1}
By Lemma \ref{lem13} the proof of Theorem \ref{thm:red2}  for $n=1$, and
hence of Theorem \ref{thm:red}, follows as soon as we have shown that \eqref{secondcase}
implies \eqref{n4}. 
To that end, consider the 1D case $n=1$, with $\Omega=(0,1)$, $\bb = 1$, and $c$ piecewise constant.
 
 Recalling that 
\be
\label{eg}
g= (\bar{w}^\delta_\tria)'+c \bar{u}_\tria^\delta-f, \quad e=\bar{u}_\tria^\delta-\bar{w}^\delta_\tria, 
\ee 
the definition of $(\bar{u}^\delta_\tria,\bar{w}^\delta_\tria)$ as minimizer of $\eta_\Omega^2(\,,\,;f)$ over $\breve{\U}^\delta_\tria$
shows that
$$
\langle u-w,e\rangle_{L_2(\Omega)}+\langle  w' +c u,g\rangle_{L_2(\Omega)}=0 \quad((u,w) \in \breve{\U}_\tria^\delta),
$$
 or, equivalently, 
\begin{alignat}{2} \label{ortheg}
e+cg \perp_{L_2(K')} \cP_{m_u}(K')&& \quad&(K' \in \tria),
\intertext{and} \label{ortheg2}
\int_{\Omega} g w'-w e\,dx=0&& \quad&((0,w) \in \breve{\U}_\tria^\delta).
\end{alignat}

\begin{remark}
\label{rem:typeII}
When $m_u=m_w$ \eqref{ortheg} says that $e=-cg$ which means that all cells are of Type-(II). In particular, when 
in addition $c=0$ we obtain $\bar{u}_\tria^\delta=\bar{w}^\delta_\tria$.
\end{remark}

For the piecewise constant function
\be
\label{F}
F {=F(\GG,\bar {\mathcal M}_{\rm II})} :=\left\{\begin{array}{ll} \GG|_{K'}, & \text{on }K' \in \bar {\mathcal M}_{\rm II}, \\ 0 ,&  {\text{elsewhere}} ,
\end{array}\right.
\ee
let $z$ be the solution of
\be
\label{zprime}
z'=-c z+F \text{ on } (0,1),\quad z(0)=0,
\ee
i.e., $z(x)=\int_0^x F(t) e^{-\int_t^x c(\tau)\,d\tau} \,dt$.
Then
$$
 \max(\|z\|_{L_2(0,1)},\|z'\|_{L_2(0,1)}) \lesssim \|F\|_{L_2(0,1)} \lesssim  {\sqrt{\sum_{K' \in \bar {\mathcal M}_{\rm II}}\|g\|^2_{L_2(K')}}}.
$$
 Moreover, $z$ is piecewise smooth w.r.t. $\tria$, and $(z|_{K'})''=-c|_{K'} (z|_{K'})'$ ($K' \in \tria$).

Let $(0,w) \in \breve{\U}_\tria$ be defined by taking $w$ as the continuous piecewise linear interpolant of $z$ w.r.t. $\tria$.
We have that
\begin{align*}
\|z-w\|_{L_2(K')} &\lesssim \diam(K') \|z'\|_{L_2(K')},\\
\|z'-w'\|_{L_2(K')} &\lesssim \diam(K') \|z''\|_{L_2(K')} \lesssim |c|_{K'}| \diam(K') \|z'\|_{L_2(K')}.
\end{align*}

Let us first assume that $c|_{K'} \neq 0$ for all $K' \in \tria$. Using  \eqref{ortheg2}, the definition of $F$, \eqref{ortheg}, $m_w \leq m_u+1$, $F|_{K'} \in P_0(K') \subset \cP_{m_u}(K')$, and the definition of $z$, we obtain
\begin{align*}
&\sum_{K' \in \bar {\mathcal M}_{\rm II}}\|\GG\|_{L_2(K')}^2=
\sum_{K' \in \bar {\mathcal M}_{\rm II}} \int_{K'} \GG g \,dx-\int_{\Omega} g w'-w e\,dx\\
&= \sum_{K' \in \tria} \int_{K'} F g-g w' + w e\,dx= \sum_{K' \in \tria} \int_{K'} \frac{e}{c}(w'-F + c w)\,dx\\
&= \sum_{K' \in \tria} \int_{K'} \frac{e}{c}((w-z)'+ c (w-z))\,dx= \sum_{K' \in \tria} \int_{K'} \frac{e}{c}(w-z)'+ e (w-z)\,dx\\
& \lesssim \max_{K' \in \tria} \diam(K') \sum_{K' \in \tria} \|e\|_{L_2(K')} \|{z'}\|_{L_2(K')} \leq { \sigma \sqrt{\sum_{K' \in \tria}\|e\|^2_{L_2(K')}}\,\|z'\|_{L_2(0,1)}}\\
&
 \lesssim \sigma \eta_{\Omega}(\bar{u}_\tria^\delta,\bar{w}_\tria^\delta;f)^2 \lesssim \sigma \sum_{K' \in \bar {\mathcal M}_{\rm II}} \|g\|_{L_2(K')}^2,
\end{align*}
where the last inequality follows from 
\eqref{secondcase} and  {Lemma~\ref{lem1}.}

Now consider the case that for one or more $K'$, $c|_{K'} = 0$. Then on such a $K'$, $z$ is linear (or even constant when $K' \in \tria \setminus \bar {\mathcal M}_{\rm II}$) and so coincides with $w$. Let $\bar{z}$ denote the average of $z$ on $K'$.
For such a $K'$, from $e\perp \cP_{m_u}(K')$ we estimate
\begin{align*}
\Big|\int_{K'} F g-g w' + w e\,dx\Big| &=\Big|\int_{K'}  z e\,dx\Big|=\Big|\int_{K'}  (z-\bar z) e\,dx\Big| \\
&\lesssim \diam(K') \|{z'}\|_{L_2(K')} \|e\|_{L_2(K')},
\end{align*}
and arrive at the same conclusion.
 For $n=1$ this completes the proof that, for $\sigma$ sufficiently small,  \eqref{secondcase} implies \eqref{n4},
 and thus of Theorem \ref{thm:red2}. Note that $\alpha>0$ could even be stipulated as small as we wish.\hfill $\Box$

\subsection{Theorem \ref{thm:red2} for $n>1$}\label{ssec:proofn}
 The above reasoning for $n=1$ does not seem to directly carry over to the multi-dimensional case. 
In fact, it is not clear how to approximate the solution $z$ to the analog of \eqref{zprime} by a 
$w$-component in the current trial space, the difficulty being the non-smoothness of $z$ in the directions orthogonal to $\bb$.

To deal with this problem, for $n>1$ we consider a downwind enriched refinement procedure as specified in Definition~\ref{def:effective}.
Let us assume that nevertheless Theorem~\ref{thm:red2} does \emph{not} hold.
That is,  there is a $\vartheta \in (0,1]$ such 
for any $\nu<1$, $r \in \N$,
 there exist $\tria \in \T$, $f \in \F_\tria^\delta$ with the property that for 
the marked cells $\bar{\cM}=\bar{\cM}((\bar{u}_\tria^\delta,\bar{w}_\tria^\delta),\vartheta)$ and refined triangulation
$\tTr=\tTr(\tria,\bar{\cM},r)$, one still has
\be
\label{still}
\eta_\Omega (\bar u^\delta_{\tTr},\bar w^\delta_{\tTr};f)> \nu \eta_\Omega (\bar u^\delta_\Tr,\bar w^\delta_\Tr;f).
\ee

Splitting $\bar{\mathcal M}=\bar{\mathcal M}_{\rm I} \dot\cup \bar{\mathcal M}_{\rm II}$ as before, as we have seen in Sect.~\ref{ssec:selection}
for $\nu$ sufficiently close to $1$ and $r$  sufficiently large, the case that
$\eta_{\bar {\mathcal M}_{\rm II}}(\bar{u}_\tria^\delta,\bar{w}_\tria^\delta;f)^2 < \eta_{\bar {\mathcal M}_{\rm I}}(\bar{u}_\tria^\delta,\bar{w}_\tria^\delta;f)^2$
would, on account of Corollary~\ref{lem:muc},  immediately lead to a contradiction. 
 
So let us focus on the case that 
\be \label{interestingcase}
\eta_{\bar {\mathcal M}_{\rm II}}(\bar{u}_\tria^\delta,\bar{w}_\tria^\delta;f)^2 \geq  \eta_{\bar {\mathcal M}_{\rm I}}(\bar{u}_\tria^\delta,\bar{w}_\tria^\delta;f)^2.
 \ee
 Following the analysis of the previous subsection \S\ref{ssec:proof1}, recall the definitions of 
$$
g= \partial_\bb\bar{w}^\delta_\tria+c \bar{u}_\tria^\delta-f, \quad e=\bar{u}_\tria^\delta-\bar{w}^\delta_\tria, 
$$
and that of $G$ in \eqref{defG} and $F$ in \eqref{F}.
From the definition of bulk chasing, \eqref{interestingcase}, and Lemma~\ref{lem1} we infer that
\be \label{bound_on_g}
\begin{split}
\eta_\Omega(\bar{u}_\tria^\delta,\bar{w}_\tria^\delta;f) &\leq {\textstyle \frac{\sqrt{2}}{\vartheta}} 
\eta_{\cup\{K' \in \bar {\mathcal M}_{\rm II}\}}(\bar{u}_\tria^\delta,\bar{w}_\tria^\delta;f)
\\
&\leq {\textstyle \frac{\sqrt{2}}{\vartheta}} 
\sqrt{(2\|c\|_{L_\infty(\Omega)}+1)^2+1}\,
\|g\|_{L_2(\cup\{K' \in \bar {\mathcal M}_{\rm II}\})}.
\end{split}
\ee

Let us now define the quantities $\tilde g$, $\tilde e$, and $\tilde G$ in analogy to $g, e$, and $G$, but with respect to
 the least-squares solution $(\bar{u}_{\tTr}^\delta,\bar{w}_{\tTr}^\delta) \in \breve{\U}_{\tTr}^\delta$ and  the refined partition $\tTr$.
The pair $(\bar{u}_{\tTr}^\delta,\bar{w}_{\tTr}^\delta)$ being a minimizer of $\eta_\Omega^2(\,,\,;g)$ over $\breve{\U}_{\tTr}^\delta$ is equivalent to
\be \label{orthogonality}
\tilde{e}+c \tilde {g} \perp_{L_2(\tilde K)} \cP_{m_u}(\tilde K) \,\,(\tilde K \in \tTr), \quad \int_\Omega \tilde{g} \partial_\bb w-w \tilde{e}\,dx=0 \,\,((0,w) \in \breve{\U}_{\tTr}^\delta).
\ee

As shown next, the assumption that the error indicator has not been reduced much when passing to $\tTr$, implies that $g,e$
 must be very close to $\tilde g, \tilde e$, respectively.  
 In fact, the orthogonality relation analogous to \eqref{ds1} reads as
$$
\eta_\Omega^2(\bar u^\delta_{\tTr},\bar w^\delta_{\tTr}; f) = \eta_\Omega^2(\bar u^\delta_\Tr,\bar w^\delta_\Tr; f)
- \eta_\Omega^2(\bar u^\delta_{\tTr}-\bar u^\delta_\Tr,\bar w^\delta_{\tTr}- \bar w^\delta_\Tr; 0).
$$
In combination with \eqref{bound_on_g} and our assumption \eqref{still}, this shows that there exists a $\zeta=\zeta(\nu)$ with $\lim_{\nu \uparrow 1} \zeta(\nu)=0$ such that
\be \label{est3}
\|g-\tilde g\|_{L_2(\Omega)} \leq \zeta \|g\|_{L_2(\cup\{K' \in \bar {\mathcal M}_{\rm II}\})},\quad \|e-\tilde e\|_{L_2(\Omega)} \leq \zeta \|g\|_{L_2(\cup\{K' \in \bar {\mathcal M}_{\rm II}\})}.
\ee
This fact together with an affirmative answer to the following conjecture will allow us to complete the Proof of Theorem \ref{thm:red2}.

 \begin{conjecture} \label{conjecture1} There exist constants $\xi<\Big({\textstyle \frac{\sqrt{2}}{\vartheta}} 
\sqrt{(2\|c\|_{L_\infty(\Omega)}+1)^2+1}\Big)^{-1}$ and $r \in \N$, such that there exists a $(0,\tilde w) \in \breve{\U}_{\tTr}^\delta$ with
 \be
 \label{relative}
 \| \partial_\bb \tilde w +c \tilde w-F\|_{L_2(\Omega)} \leq \xi \|F\|_{L_2(\Omega)}, \quad\|\tilde w\|_{L_2(\Omega)} \lesssim \|F\|_{L_2(\Omega)},
 \ee
where $\tilde w$ vanishes outside the union of the cells of $\tria$ that were refined in $\tTr=\tTr(\tria,\bar{\cM},r)$.
 \end{conjecture}

 We postpone supporting arguments for the validity of this conjecture and turn first, for $r$ large enough and $\nu$ sufficiently close to $1$,
to verifying the hypothesis of Lemma~\ref{lem13}. This lemma then asserts the validity of Theorem~\ref{thm:red2}, which, for $\nu$ sufficiently close
to $1$, will contradict \eqref{still}, thereby finishing the proof.

To that end, with $\tilde w$ from Conjecture~\ref{conjecture1}, using \eqref{orthogonality} we write
\begin{align*}
&\|G\|^2_{L_2(\cup\{K' \in \bar {\mathcal M}_{\rm II}\})}
=
\sum_{K' \in \bar {\mathcal M}_{\rm II}} \int_{K'} G g \, dx\\ 
&=\sum_{K' \in \bar {\mathcal M}_{\rm II}} \int_{K'} G \tilde g \, dx+\sum_{K' \in \bar {\mathcal M}_{\rm II}} \int_{K'} G (g-\tilde g) \, dx\\
&= \int_\Omega F \tilde g \, dx-\int_\Omega \tilde g \partial_\bb \tilde w-\tilde w \tilde e\,dx+\sum_{K' \in \bar {\mathcal M}_{\rm II}} \int_{K'} G (g-\tilde g) \, dx\\
&= -\int_\Omega ( \partial_\bb \tilde w+c \tilde w-F)\tilde g \, dx+\sum_{K' \in \bar {\mathcal M}_{\rm II}} \int_{K'} G (g-\tilde g) \, dx+\int_\Omega \tilde w(\tilde e+c \tilde g)\dx.
\end{align*}
The first and second term on the right can be bounded by
\be \label{est1}
\begin{split}
\big|&\int_\Omega ( \partial_\bb \tilde w+c \tilde w-F)\tilde g \, dx\big| \leq \xi \|F\|_{L_2(\Omega)}(1+\zeta)\|g\|_{L_2(\Omega)}\\
& \leq \xi \|G\|_{L_2(\cup\{K' \in \bar {\mathcal M}_{\rm II}\})} (1+\zeta)  {\textstyle \frac{\sqrt{2}}{\vartheta}} \sqrt{(2\|c\|_{L_\infty(\Omega)}+1)^2+1} \,
\|g\|_{L_2(\cup\{K' \in \bar {\mathcal M}_{\rm II}\})},
\end{split}
\ee
where we have used  \eqref{est3} and \eqref{bound_on_g}, and
\be \label{est2}
\big|\sum_{K' \in \bar {\mathcal M}_{\rm II}} \int_{K'} G (g-\tilde g) \, dx\big| \leq \zeta \|G\|_{L_2(\cup\{K' \in \bar {\mathcal M}_{\rm II}\})} \|g\|_{L_2(\cup\{K' \in \bar {\mathcal M}_{\rm II}\})},
\ee
respectively.

To proceed let  $Q_{\tTr}$ denote the $L_2(\Omega)$-orthogonal projector onto $\prod_{\tilde K\in \tTr} \cP_{m_u}(\tilde K)$, using \eqref{orthogonality} for the third term we write 
\begin{align*}
\int_\Omega \tilde w(\tilde e+c \tilde g)\dx&=
\int_\Omega \big((I-Q_{\tTr})\tilde w\big)(\tilde e +c \tilde g)\dx\\
&=
\int_\Omega \big((I-Q_{\tTr})\tilde w\big)\big(\tilde e-e+c (\tilde g-g)\big)\dx\\
&\qquad +\int_\Omega \tilde w \big((I-Q_{\tTr})(e+c g)\big)\dx.
\end{align*}
Thanks to \eqref{est3}, the first term at the right can be bounded by a constant multiple of $\zeta \|G\|_{L_2(\cup\{K' \in \bar {\mathcal M}_{\rm II}\})} \|g\|_{L_2(\cup\{K' \in \bar {\mathcal M}_{\rm II}\})}$.
We use next that $w$ vanishes outside the union of the cells of $\tria$ which have been refined in $\tTr=\tTr(\tria,\bar{\cM},r)$, and that $e$ and $g$ are piecewise polynomial w.r.t. $\tria$. Moreover,  by Remark \ref{rem:coarse}, all cells in the support of $\tilde w$
are (at least) $r$th refinements of cells in $\Tr$ underlying $e$ and $g$.  Hence, the usual combination of direct and inverse estimates shows then that the second term can be bounded by
$\|G\|_{L_2(\cup\{K' \in \bar {\mathcal M}_{\rm II}\})} \eta(r)  \|g\|_{L_2(\cup\{K' \in \bar {\mathcal M}_{\rm II}\})}$, where $\eta$ as a function of $r$,
tends to zero as $r\to \infty$. 
For any constant $\alpha \in \Big(\xi  {\textstyle \frac{\sqrt{2}}{\vartheta}} \sqrt{(2\|c\|_{L_\infty(\Omega)}+1)^2+1},1\Big)$, the combination of these latter results, \eqref{est1}, and \eqref{est2} shows that for $r$ large enough and $\nu$ sufficiently close to $1$, $\|G\|_{L_2(\cup\{K' \in \bar {\mathcal M}_{\rm II}\})}\leq \alpha \|g\|_{L_2(\cup\{K' \in \bar {\mathcal M}_{\rm II}\})}$, which by Lemma~\ref{lem13}, for $\nu$ sufficiently close to $1$, contradicts \eqref{still}, as required.\hfill $\Box$\\

Let us close this section with some brief comments on Conjecture \ref{conjecture1}. First, 
as mentioned earlier, by Remark \ref{rem:coarse}, the downwind enrichment in the refinement strategy
makes sure that  the correction $\tilde w$ is constructed on (an essentially uniform) refined mesh.
This certainly helps a relation like \eqref{relative} to be possible and actually motivated the inclusion of the downwind enrichments.
Moreover,  the conjecture asks ``only'' for a fixed relative accuracy $\xi$ where
$\xi$ need not be arbitrarily small. Given that the data are piecewise polynomials (which are actually piecewise constants in stream direction),
this does not seem to ask for too much. 

On the other hand, since we can neither limit a priori the number of polynomial pieces in $F$ nor
their position relative to the direction of $\bb$ an argument does not seem to be straightforward. In fact, whereas    we can represent the
exact solution of $\partial_\bb z+cz=F$ with zero inflow conditions explicitly along characteristics ensuring sufficient smoothness in this direction,
smoothness in cross-flow direction does not seem to be easy to control. Nevertheless, the overall variation in cross-flow direction is still
highly restrained for data of the type $F$. 


Finally, we would like to stress that a possibly $\Tr$-dependent $r$ such that \eqref{relative} holds true always exists.
By the above arguments this immediately translates 
into a statement on error reduction based on such a (variable) refinement depth.

\section{Concluding Remarks} \label{sec:6}
We have established reliability and efficiency of computable local error indicators for DPG discretizations of linear
transport equations with variable convection and reaction coefficients. For constant (with respect to the spatial variables) convection fields,
arising for instance in kinetic models, we have determined refinement strategies based on the a posteriori error indicators which are guaranteed
to give rise to a fixed error reduction rate. The latter results make essential use of a tight interrelation of the DPG scheme with certain
least squares formulations providing insight of its own right. In particular, error reduction for one scheme implies the same for the other one.
To our knowledge the issue of error reduction for least squares methods even for the classical elliptic case is largely open.
In that sense the present results mark some progress in this regard as well.
 
On the other hand, in view of these findings one may raise the question as to why not using the seemingly simpler least-squares scheme instead of the DPG scheme.
However, giving up on the simple interpretation of the $w$-component as a second approximation to the exact solution in a stronger norm when $f\in L_2(\Omega)$, the DPG scheme still provides a meaningful approximate solution $\uPG$ in $L_2(\Omega)$ to the transport equation even when
$f$ is less regular. 
But also for $L_2$-data $f$, in the least squares formulation errors are measured solely in a norm that  depends in a very sensitive way
directly on the convection field. In the variable convection case the corresponding space varies  essentially (even as a set) under perturbations of
this convection field. Therefore, at this point Proposition \ref{solsclose} serves primaily as a theoretical tool.

Among other things a prize for using the interrelation between DPG and \new{least} squares formulations is a remaining lack of quantification
of the error reduction results manifesting itself in two ways: the subgrid depth needed to establish efficiency and reliability of the
computable error indicators, similar to establishing uniform inf-sup stability of the pairs of trial- and  test-spaces, is not precisely specified.
As indicated by earlier numerical results in \cite{35.8566} any attempt along the given lines to quantify the subgrid-depth would still
be over pessimistic. The same is expected to be true for the refinement depth $\refs$ associated with the marked cells.
These issues call for further research in this area.

Finally, the refinement strategies that can be shown to guarantee a fixed error reduction involve for several spatial variables so far
a certain downstream enrichment of the marked cells in combination with a conjecture.
It is open whether this enrichment is in general necessary which would establish an essential difference from the univariate case
where it is not necessary.



\begin{thebibliography}{DHSW12}

\bibitem[BDD04]{BDD04}
P. Binev, W. Dahmen, and R. DeVore, Adaptive finite element methods with convergence rates,  {\em Numer. Math.}, 97:219--268, 2004.


\bibitem[BDS17]{35.8566}
D.~Broersen, W.~Dahmen, and R.P. Stevenson.
\newblock On the stability of {DPG} formulations of transport equations.
\newblock {\em Math. Comp.}, 87(311):1051--1082, 2018.

\bibitem[BX91]{34.55}
J.H. Bramble and J.~Xu.
\newblock Some estimates for a weighted \({L}^2\) projection.
\newblock {\em Math. Comp.}, 56:463--476, 1991.

\bibitem[CDG16]{37.4}
C.~Carstensen, L.~Demkowicz, and J.~Gopalakrishnan.
\newblock Breaking spaces and forms for the {DPG} method and applications
  including {M}axwell equations.
\newblock {\em Comput. Math. Appl.}, 72(3):494--522, 2016.

\bibitem[CDW12]{CDW12}
A.~Cohen, W.~Dahmen, G.~Welper, 
\newblock Adaptivity and Variational Stabilization for Convection-Diffusion Equations, 
\newblock {\em ESAIM: Mathematical Modelling and Numerical Analysis},   46(5): 1247--1273, 2012.

\bibitem[DHSW12]{58.3}
W.~Dahmen, C.~Huang, Ch.~Schwab, and G.~Welper.
\newblock Adaptive {P}etrov-{G}alerkin methods for first order transport
  equations.
\newblock {\em SIAM J. Numer. Anal.}, 50(5):2420--2445, 2012.

\bibitem[GQ14]{75.61}
J.~Gopalakrishnan and W.~Qiu.
\newblock An analysis of the practical {DPG} method.
\newblock {\em Math. Comp.}, 83(286):537--552, 2014.

\bibitem[NSV09]{NSV09}
R. H. Nochetto, K. G. Siebert, and A Veeser, Theory of adaptive finite element methods: An introduction. 
In Ronald DeVore and Angela Kunoth, editors, {\em Multiscale, Nonlinear and Adaptive Approximation}, pp. 409--542. Springer Berlin Heidelberg, 2009.

\bibitem[S07]{S07}
R. P. Stevenson, Optimality of a standard adaptive finite element method, \emph{Found. Comput. Math.}, 7(2): 245--269, 2007.


\bibitem[Ver96]{307}
R.~Verf\"{u}rth.
\newblock {\em A Review of A Posteriori Error Estimation and Adaptive
  Mesh-Refinement Techniques}.
\newblock Wiley-Teubner, Chichester, 1996.

\end{thebibliography}

\end{document}